\documentclass[11pt,oneside]{article}
\usepackage{amssymb,amsmath}
\usepackage{url}
\renewcommand{\le}{\leqslant}
\renewcommand{\ge}{\geqslant}
\renewcommand{\Re}{\mathrm{Re}}

\newcommand{\dist}{\mathbf{dist}}
\newcommand{\diam}{\mathrm{diam}\,}
\newcommand{\covol}{\mathrm{covol}}

\newcommand{\Span}{\mathrm{span}}

\newcommand{\RR}{\mathbb{R}}
\newcommand{\ZZ}{\mathbb{Z}}
\newcommand{\QQ}{\mathbb{Q}}

\newcommand{\NN}{\mathbb{N}}

\newcommand{\CC}{\mathbb{C}}

\newcommand{\LLL}{\mathcal{L}}
\newcommand{\VVV}{\mathcal{V}}

\newcommand{\CCC}{\mathcal{C}}

\newcommand{\SSS}{\mathcal{S}}
\newcommand{\HHH}{\mathcal{H}}
\newcommand{\PPP}{\mathcal{P}}

\newcommand{\MMM}{\mathcal{M}}

\newcommand{\vx}{\mathbf{x}}
\newcommand{\vu}{\mathbf{u}}
\newcommand{\vv}{\mathbf{v}}

\newcommand{\vp}{\mathbf{p}}
\newcommand{\vq}{\mathbf{q}}

\newcommand{\vb}{\mathbf{b}}
\newcommand{\va}{\mathbf{a}}
\newcommand{\vf}{\mathbf{f}}
\newcommand{\vkappa}{\boldsymbol{\kappa}}

\newtheorem{lemma}{Lemma}
\newtheorem{theorem}{Theorem}
\newtheorem{proposition}{Proposition}

\newtheorem{corollary}{Corollary}

\newtheorem{theorembb}{Theorem BB}

\newenvironment{proof}{\textsc{Proof. }}{\ \newline\hspace*{\fill}$\boxtimes$}

\newcommand{\vxi}{\boldsymbol{\xi}}

\newcommand{\vdelta}{\boldsymbol{\delta}}
\newcommand{\vepsilon}{\boldsymbol{\epsilon}}

\newcommand{\SL}{\mathrm{SL}}
\newcommand{\vol}{\mathrm{Vol}}
\newcommand{\codim}{\mathrm{codim}\;}
\newcommand{\res}{\mathrm{Res}}

\begin{document}
\title{Simultaneous Diophantine approximation to points on the Veronese curve}

\author{Dmitry Badziahin}


\maketitle

\begin{abstract}
We compute the Hausdorff dimension of the set of simultaneously
$q^{-\lambda}$-well approximable points on the Veronese curve in $\RR^n$
for $\lambda$ between $\frac{1}{n}$ and $\frac{2}{2n-1}$. For $n=3$, the
same result is given for a wider range of $\lambda$ between $\frac13$ and
$\frac12$. We also provide a nontrivial upper bound for this Hausdorff
dimension in the case $\lambda\le \frac{2}{n}$. In the course of the proof
we establish that the number of cubic polynomials of height at most $H$ and
non-zero discriminant at most $D$ is bounded from above by $c(\epsilon)
H^{2/3 + \epsilon} D^{5/6}$.
\end{abstract}

{\footnotesize{{\em Keywords}: Diophantine exponents, simultaneously well
approximable points, Veronese curve, Hausdorff dimension, simultaneous
Diophantine approximation on manifolds, cubic polynomials of bounded
discriminant

Math Subject Classification 2020: 11J13, 11J54, 11J82, 11K55}}

\section{Introduction}

For a positive real number $\lambda$ the set $S_n(\lambda)$ of
$q^{-\lambda}$-well approximable points in $\RR^n$ is defined as follows:
$$
S_n(\lambda):= \{\vx\in\RR^n: ||q\vx - \vp||_\infty < q^{-\lambda}\;\mbox{ for i.m. } (q,\vp)\in \ZZ^{n+1}\}.
$$
One of the topical problems in the metric theory of Diophantine approximation
is to understand the structure of the intersection of $S_n(\lambda)$ with a
suitable manifold $\MMM$, see for example~\cite{kle_mar_1998,
beresnevich_2012, ber_yan_2023} where this problem is investigated. In the
landmark paper~\cite{kle_mar_1998} Kleinbock and Margulis established that
for all so-called nondegenerate manifolds $\MMM$ and all $\lambda>1/n$, the
set $S_n(\lambda)\cap \MMM$ has zero Lebesgue measure. But regarding its
Hausdorff dimension, much less is known. Beresnevich~\cite{beresnevich_2012}
showed that for $\lambda$ close enough to $1/n$ and nondegenerate $\MMM$,
$$
\dim (S_n(\lambda)\cap\MMM) \ge \frac{n+1}{\lambda+1} - \codim\MMM.
$$
For the case when $\MMM$ is a nondegenerate curve, the above
inequality is attained for $\frac{1}{n}\le\lambda\le
\frac{3}{2n-1}$. Later, Beresnevich and Yang~\cite{ber_yan_2023}
proved that this bound is sharp for $\lambda$ in close proximity of
$1/n$, much closer than $\frac{3}{2n-1}$. The last improvement of
their bound on $\lambda$ can be found in \cite{sst_2024}. With
respect to curves, it states as follows. Suppose that $\CCC$ is
parametrised by a $n$ times continuously differentiable function
$\vf: J\to \RR^n$ such that its derivatives up to degree $n$ at any
$x\in J$ span $\RR^n$. Then
\begin{equation}\label{sst}
\dim (S_n(\lambda)\cap\CCC) = \frac{2-(n-1)\lambda}{1+\lambda}\qquad
\forall\; \lambda \in \left[\frac1n,
\frac{1}{n}+\frac{n+1}{n(2n-1)(n^2+n+1)}\right).
\end{equation}

For larger values of $\lambda$ the structure of the set
$S_n(\lambda)\cap \MMM$ is mostly unknown. However we know for sure
that for large enough $\lambda$ it substantially depends on a
manifold $\MMM$. For example, consider the circle $\CCC:=
\{\vx\in\RR^2: x_1^2+x_2^2=3\}$. It is not too difficult to verify
that for $\lambda>1$ and $\vx\in \CCC$, the inequality $||q\vx
-\vp||_\infty < q^{-\lambda}$, $(q,\vp)\in\ZZ^3\setminus \mathbf{0}$
for large $q$ implies that $\vp/q\in\CCC$. Then we immediately
deduce that $S_2(\lambda)\cap \CCC = \emptyset$. On the other hand,
for the Veronese curve $\VVV_n:=\{(x,x^2,\ldots, x^n): x\in\RR\}$
and all $\lambda>1$ Schleischitz~\cite{schleischitz_2016} showed
that $\dim (S_n(\lambda)\cap \VVV_n) = \frac{2}{n(1+\lambda)}$, in
sharp contrast with the previous example.

In this paper, we investigate the Hausdorff dimension of
$S_n(\lambda)\cap \VVV_n$. In the literature, this set is often
considered from a different perspective. For $x\in\RR$, by $n$'th
simultaneous Diophantine exponent $\lambda_n(x)$ we define the
supremum of all $\lambda$ such that the inequality
$$
\max_{1\le i\le n} ||qx^i|| \le q^{-\lambda}
$$
has infinitely many solutions $q\in\ZZ$. Here and throughout the
paper, by $||x||$ we mean the distance from $x\in \RR$ to the
nearest integer. We refer to \cite{bugeaud_2014} for an overview of
the known results about the exponents $\lambda_n(x)$. One can easily
see that the set of $x\in\RR$ with $\lambda_n(x)\ge \lambda$
coincides with the projection of $S_n(\lambda)\cap \VVV_n$ to the
first coordinate axis. Therefore the Hausdorff dimensions of
$S_n(\lambda)\cap \VVV_n$ and $\{x\in \RR: \lambda_n(x)\ge
\lambda\}$ coincide. In this paper, we will often work with the
latter set. In particular, for an open interval $I\subset \RR$ we
define
$$
S_n(I,\lambda):= \{x\in I: \lambda_n(x)> \lambda\}.
$$

One can define analogues of $S_n(I,\lambda)$ for general curves $\CCC$. Let
$\CCC$ be two times continuously differentiable. By permuting the coordinates
in $\CCC$ if needed, locally it can be parametrised as $\vx =
\vf(x):=(x,f_2(x),\ldots, f_n(x))$, $x\in J\subset \RR$ where $f_i(x)\in
C^2(J)$. For the moment, we only put the additional condition that
$f_i(x)\neq 0$ for all $1\le i\le n$ and almost all $x\in J$ in terms of the
Lebesgue measure. Later, we will impose more conditions on $\vf$. For $x\in
J$, we define the Diophantine exponent $\lambda^\vf (x)$ as the supremum of
all $\lambda$ such that the inequality $||q\vf(x)-\vp||_\infty \le
q^{-\lambda}$ has infinitely many solutions $(q,\vp)\in \ZZ^{n+1}$. Next, for
$I\subset J$ we define
$$
S_n^\vf (I,\lambda):= \{x\in I: \lambda^\vf (x)> \lambda\}.
$$
Then, as for $S_n(I,\lambda)$, the set $S_n^\vf(I,\lambda)$ is the projection
of $S_n(\lambda)\cap \CCC_I$ to the first coordinate axis, where $\CCC_I$ is
the curve $\CCC$ restricted to the domain $I$. Hence we have $\dim
S_n^\vf(I,\lambda) = \dim(S_n(\lambda)\cap \CCC_I)$.

In the paper we use the Vinogradov symbol: for positive values $A$
and $B$, $A\gg B$ means that $A\ge cB$ where $c>0$ is a constant.
Throughout the paper, $c$ may depend on $\vf$ and $I$ but does not
depend on $x\in I$. In a similar way we define $A\ll B$ and $A\asymp
B$, the latter means that $A\ll B$ and $A\gg B$ at the same time.

By fixing an appropriate subinterval $I\subset J$, without loss of generality
we assume
\begin{equation}
\max_{x\in I} \{||\vf(x)||_\infty, ||\vf'(x)||_\infty,
||\vf''(x)||_\infty\} \ll 1.
\end{equation}
Also without loss of generality, we may assume that $I$ is separated
from zeroes of $f_i$, i.e. for all $x\in I$, $0\le i\le n$,
$|f_i(x)|\gg1$. In that case, for any integer point
$(q,\vp)\in\ZZ^{n+1}$ with $q$ large enough such that $||q\vf(x) -
\vp||_\infty \le q^{-\lambda}$, we have $|q|\gg ||(q,\vp)||_\infty$.
Notice that in the case of the Veronese curve this condition means
that $I$ is separated from 0. However, for the purposes of the
computation of $\dim S_n(I,\lambda)$ this condition can be lifted.
Indeed, by choosing a sequence of positive $\epsilon_i$,
$i\in\ZZ^+$, we get
$$
S_n(I,\lambda) = \bigcup_{i\in \ZZ^+} S_n(I\setminus [-\epsilon_i,
\epsilon_i], \lambda) \cup \big(S_n(I,\lambda)\cap
\{\mathbf{0}\}\big).
$$
Therefore if the bounds on $\dim S_n(I,\lambda)$, for $I$ separated
from zero, do not depend on $I$ the same bounds are in place for all
intervals $I$.

The sets $S_n(I,\lambda)$ were extensively investigated. But despite all the
efforts, their Hausdorff dimension is only completely known for $n=1$ and
$2$. Namely, the classical theorem of Jarnik and Besicovich states that
$$
\dim S_1(I,\lambda) = \frac{2}{1+\lambda}\quad \forall \lambda\ge 1.
$$
In the case $n=2$ the result is more complicated and is achieved by joint
efforts of Beresnevich, Dickinson, Schleischitz, Vaughan and
Velani~\cite{bdv_2007,vau_vel_2006,schleischitz_2016}:
$$
\dim S_2(I,\lambda) = \left\{\begin{array}{rl}
\frac{2-\lambda}{1+\lambda}&\mbox{ if }\frac12 \le \lambda\le
1;\\[1ex]
\frac{1}{1+\lambda}&\mbox{if } \lambda>1.
\end{array}\right.
$$
In fact, the first part of this formula for $\frac12 \le \lambda\le
1$ was verified for all sets $S_2^\vf (I,\lambda)$ with $\vf$ having
non-vanishing curvature. The second part of the result is specific
to $S_2(I,\lambda)$.

For $n\ge 3$, the Hausdorff dimension of $S_n(I,\lambda)$ is only
known for values of $\lambda$ in the close proximity of $1/n$, as
was already mentioned in~\eqref{sst}, and for relatively large
$\lambda$, due to the work of the author and
Bugeaud~\cite{bad_bug_2020}. Namely, for $\lambda\ge
\frac{n+4}{3n}$,
$$
\dim S_n(I,\lambda) = \frac{2}{n(1+\lambda)}.
$$
This result is a corollary of a more general upper and lower bounds
on $\dim S_n(I,\lambda)$ that cover a bigger range of $\lambda$:

\begin{theorembb}[Badziahin, Bugeaud, 2020]\label{th_bb}
For all $0\le k\le n-1$ and $\lambda\ge 1/n$ one has
$$
\dim S_n(I,\lambda)\ge \frac{(k+2)(1-k\lambda)}{(n-k)(1+\lambda)}.
$$
On the other hand, let $\lambda\ge \big\lfloor\frac{n+1}{2}\big\rfloor^{-1}$.
Setting, $m = \lfloor1/\lambda\rfloor$, one has
$$
\dim S_n(I,\lambda)\le \max_{0\le k\le
m}\frac{(k+2)(1-k\lambda)}{(n-2k)(1+\lambda)}.
$$
\end{theorembb}
The essence of this theorem is that $\dim S_n(I,\lambda)$ is
sandwiched between two piecewise rational functions, where the upper
bound is only provided for relatively large $\lambda$. Those two
bounds coincide for $\lambda\ge (n+4)/(3n)$.

The exact value of $\dim S_n(I,\lambda)$ for intermediate values of
$\lambda$ remains a mystery. It is believed to be a piecewise
rational function but there is no even a guess regarding its precise
formula.

In this paper we provide tighter upper bounds for $\dim
S_n(I,\lambda)$. The first result in some sense complements
Theorem~BB for $\lambda <\big\lfloor\frac{n+1}{2}\big\rfloor^{-1}$.
To the best of the author's knowledge, no other non-trivial upper
bounds for $\lambda$ in this range and outside of the interval
from~\eqref{sst} were known before.

\begin{theorem}\label{th1}
Let $m = \big\lfloor\frac{n-1}{2}\big\rfloor$. For all $\frac{1}{n} \le
\lambda \le \frac{2}{n}$ and $0\le k\le m$ one has
\begin{equation}\label{th1_eq}
\dim S_n(I,\lambda)\le
\frac{(n-k+1)(1-k\lambda)}{(n-2k)(1+\lambda)}.
\end{equation}
\end{theorem}

One can easily check that for $k=0$ and $\lambda>\frac1n$ the right
hand side of~\eqref{th1_eq} is strictly smaller than 1. Therefore
Theorem~\ref{th1} always gives a non-trivial upper bound for $\dim
S_n(I,\lambda)$. As in Theorem~BB, this upper bound is a piecewise
rational function of $\lambda$ where different pieces are derived
from different values of $k$. An enthusiastic reader can verify that
there exist values
$$
\frac{1}{n}=\lambda_0 < \lambda_1 < \cdots < \lambda_m<\lambda_{m+1}
= \frac{2}{n}
$$
such that for $\lambda_t<\lambda<\lambda_{t+1}$ the value $k=t$
provides the smallest upper bound in~\eqref{th1_eq} among all other
values of $k$. Then by considering $\lambda=\frac{2}{n}$ and $k=m$
we derive

\begin{corollary}
$$
\dim S_n(I,2/n) \le \left\{\begin{array}{rl}
\displaystyle\frac{m+2}{2m+3}&\mbox{if }\; n=2m+1,\\[2.5ex]
\displaystyle\frac{m+3}{2m+4}&\mbox{if }\; n=2m+2.
\end{array}\right.
$$
In both cases the right hand side is of the form $\frac12 +
O(n^{-1})$.
\end{corollary}

Note that if the formula~\eqref{sst} is satisfied for $\lambda =
\frac{3}{2n-1}$, as conjectured by Beresnevich and Yang
in~\cite{ber_yan_2023}, then we already have $\dim S_n(I,3/(2n-1)) =
\frac12$. This demonstrates that most likely, the result in
Theorem~\ref{th1} is not optimal.

The second result significantly extends the interval for $\lambda$
from~\cite{sst_2024} where $\dim S_n(I,\lambda)$ is given
by~\eqref{sst}.

\begin{theorem}\label{th2}
For all $\lambda$ between $\frac1n$ and $\frac{2}{2n-1}$ one has
$$
\dim S_n(I,\lambda) = \frac{2-(n-1)\lambda}{1+\lambda}.
$$
For $n=3$ the range for $\lambda$ can be extended to $\frac13\le
\lambda \le \frac{1}{2}$.
\end{theorem}

Notice that in~\eqref{sst} the upper bound for $\lambda$ is of the
form $\frac1n + \frac{1}{2n^3} + O(n^{-4})$, while in our result it
is $\frac1n + \frac{1}{2n^2}+O(n^{-3})$.

For $n=3$ and $\frac12\le \lambda\le \frac35$ we can significantly
improve the upper bound in Theorem~\ref{th1}.

\begin{theorem}\label{th5}
For all $\lambda$ between $\frac12$ and $\frac35$ one has
$$
\dim S_3(I,\lambda) \le \frac{4-2\lambda}{3(1+\lambda)}.
$$
\end{theorem}

{\bf Remark.} The upper bound in Theorem~\ref{th5} is not sharp.
With more delicate estimates in some of the cases of the proof in
Section~\ref{sec3}, it can be improved. However we decide not to
make already tedious proof even more complicated and leave the bound
in its present form.

The proofs of Theorem~\ref{th2} for $n=3$ and Theorem~\ref{th5} rely
on counting cubic polynomials of bounded heights and discriminants.
Here by the height of a polynomial $P$, denoted by $H(P)$, we
understand its naive height, i.e. the supremum norm of its
coefficient vector. We also denote the discriminant of $P$ by
$D(P)$. In this paper, we prove the following result which is of
independent interest:

\begin{theorem}\label{th3}
For any $\epsilon>0$ there exists a constant $c = c(\epsilon)$ such
that the number $N(H,D)$ of cubic integer polynomials $P$ with
$H(P)\le H$ and $0<|D(P)|\le D$ is bounded from above by
\begin{equation}\label{th3_eq}
N(H,D)\le c H^{2/3+\epsilon} D^{5/6}.
\end{equation}
\end{theorem}

{\bf Remark.} We believe that by using more delicate arguments than
in this paper, one can remove the term $\epsilon$ in the upper
bound~\eqref{th3_eq}. However, this is not important for our main
aim and will only make the paper more complicated.

If we take $D = H^{4 - 2v}$ the bound~\eqref{th3_eq} transforms to
$N(H,D)\le c H^{4 - \frac{5}{3}v + \epsilon}$. On the other hand,
Beresnevich, Bernik and G\"otze~\cite{bbg_2016} verified that $N(H,
\gamma H^{4 - 2v}) \gg H^{4-\frac{5}{3}v}$ for a suitable absolute
constant $\gamma>0$. Therefore our upper bound is sharp (or rather
almost sharp as we have an additional small term $H^\epsilon$). It
is worth mentioning that several similar upper bounds were achieved
in the last two decades, see~\cite{kgk_2014, bbg_2017}. However, all
of them have restrictions on heights and discriminants, while our
result works for all possible pairs $(H,D)$.

Wherever possible, while proving the results of this paper we will deal with
the more general sets $S_n^\vf(I,\lambda)$ and only consider $S_n(I,\lambda)$
if we use some specific properties of the Veronese curve.

\section{Dual sets of well approximable points}

One of the ideas in this paper is to transfer the problem about
simultaneously well approximable points to the one about dually well
approximable points. Namely, a point $\vx\in\RR^n$ is called dually
$q^{-w}$-well approximable if the following inequality
$$
|\vq\cdot \vx - p|\le ||\vq||_\infty^{-w}
$$
has infinitely many solutions $(\vq,p)\in\ZZ^{n+1}$. By $D_n(w)$ we
denote the set of dually $q^{-w}$-well approximable points. For many
manifolds $\MMM$ the sets $D_n(w)\cap \MMM$ are better understood
compared to $S_n(\lambda)\cap \MMM$. For example, this is the case
for $\MMM = \VVV_n$ where the exact value of $\dim (D_n(w)\cap
\VVV_n)$ for all $n\in\NN$ and $\lambda \ge n$ was computed by
Bernik~\cite{bernik_1983}:
\begin{equation}\label{bernik}
\dim(D_n(w)\cap \VVV_n) = \frac{n+1}{w+1}.
\end{equation}
As in the simultaneous case, we locally parametrise the curve $\CCC$ by $\vf$
and consider the following set:
$$
D_n^\vf (I,w):=\{x\in I: |\vq\cdot \vf(x) - p|\le||\vq||_\infty^{-w} \mbox{ for
i.m. } (\vq,p)\in\ZZ^{n+1}\}.
$$
In the case of the Veronese curve we omit the superscript $\vf$ and
write $D_n(I,w)$.

We demonstrate the idea by proving the following result. It gives a
weaker upper bound compared to Theorem~\ref{th1} but it is satisfied
for all curves $\CCC$ under mild conditions on $\vf$. In fact, it is
a quick corollary of the classical Khintchine transference
principle~\cite[Chapter V]{cassels_1957}. But since the proof is
short, we decide to provide it here for demonstrational purposes.

\begin{theorem}\label{th4}
Suppose that for any rational hyperplane $\HHH\subset \RR^{n+1}$,
$\dim ((1,\vf(I))\cap \HHH)=0$. Then
$$
\dim S^\vf_n(I,\lambda) \le \dim D^\vf_n(I, n(1+\lambda)-1).
$$
\end{theorem}

{\bf Remark.} Notice that in Theorem~\ref{th4} the condition on the
curve $\vf(I)$ is weaker than the property of nondegeneracy as was
defined in the work of Kleinbock and Margulis~\cite{kle_mar_1998}.
Recall that a curve $\vf(I)$ is called nondegenerate at $x\in I$ if
it is enough times continuously differentiable and all partial
derivatives of $\vf$ span $\RR^n$. Then one can check that as soon
as the Hausdorff dimension of points $x\in I$ where the curve is
degenerate equals zero, the conditions of Theorem~\ref{th4} are
satisfied. On the other hand, affine lines with $f_i(x) = a_ix +
b_i$ where the numbers $a_1,a_2,\ldots, a_n$ are linearly
independent over $\QQ$, are degenerate at every point, while they
still satisfy the conditions of Theorem~\ref{th4}.

\proof Consider $x\in S^\vf_n(I,\lambda)$. By removing the set of zero
Hausdorff dimension, we make sure that $(1,\vf(x))$ does not lie on any
rational hyperplane in $\RR^{n+1}$. Then by definition, there are infinitely
many points $\vq = (q,\vp)\in\ZZ^{n+1}$ that satisfy $||q\vf(x) -
\vp||_\infty \le q^{-\lambda}$. Fix one such point $\vq$ with large enough
$q$. By the Minkowski theorem, there exists $\va\in\langle \vq\rangle^\bot$
such that $||\va||_\infty \le ||\vq||_\infty^{1/n}$.

Now recall that by the choice of $I$, $|q| \asymp ||\vq||_\infty$ and compute
\begin{equation}\label{eq1}
|\va\cdot (1,\vf(x))| = |q^{-1} (\va\cdot \vq + a_1(qf_1(x) -
p_1) + \ldots + a_n(q f_n(x) - p_n))| \ll
\frac{||a||_\infty}{||\vq||_\infty ^{1+\lambda}}.
\end{equation}
This immediately implies that $|\va\cdot (1,\vf(x))| \ll
||\va||_\infty^{1-n(1+\lambda)}$. Since the left hand side
of~\eqref{eq1} can not be equal to 0, this inequality can not be
satisfied for a fixed $\va$ and infinitely many different $\vq$.
Hence we must have the inequality $|\va\cdot (1,\vf(x))| \ll
||\va||_\infty^{1-n(1+\lambda)}$ for infinitely many
$\va\in\ZZ^{n+1}$ which immediately implies that $x\in
D^\vf_n(I,n(1+\lambda)-1)$. The statement of the theorem follows
immediately.
\endproof

For the set $S_n(I,\lambda)$, in view of Bernik's
result~\eqref{bernik}, Theorem~\ref{th4} gives an upper bound
$$
\dim S_n(I, \lambda)\le \frac{n+1}{n(1+\lambda)},
$$
which is the case $k=0$ in Theorem~\ref{th1}. To get the upper
bounds~\eqref{th1_eq} for bigger $k$, we need to use specific properties of
the Veronese curve.

Given a vector $\va\in\ZZ^{m+1}$ where $m\le n$, by $L_\va$ we
define an $m$-dimensional subspace of $\RR^{n+1}$ defined by the
following linear equations:
\begin{equation}\label{def_la}
a_0x_j + a_1x_{j+1} + \ldots + a_mx_{j+m} = 0,\quad 0\le j\le n-m.
\end{equation}

For $0\le h\le n/2$, we say that a vector $\vq\in\ZZ^{n+1}$ is of
type $h$ if $h$ is the largest number such that the rank of the
matrix
$$
M_{h,n}(\vq):=\left(\begin{array}{cccc} q_0&q_1&\cdots&q_{n-h}\\
q_1&q_2&\cdots&q_{n-h+1}\\
\vdots&\vdots&\ddots&\vdots\\
q_h&q_{h+1}&\cdots&q_n
\end{array}
\right)
$$
is $h+1$, i.e. it is full. For example, for points $\vq$ on the
Veronese curve, i.e. $\vq = (u^n, u^{n-1}v, \ldots, v^n)$, their
type equals zero.

{\bf Proof of Theorem~\ref{th1}}. Write $S_n(I,\lambda)$ as a union of
$\big\lfloor \frac{n}{2}\big\rfloor + 1$ subsets $S_n^h(I,\lambda)$, $0\le
h\le \big\lfloor \frac{n}{2}\big\rfloor$ such that
$$
S_n^h(I,\lambda) := \{x\in \RR: \max_{1\le i\le n} |q_0x^i - q_i|
\le q_0^{-\lambda}\mbox{ for i.m. } \vq\in \ZZ^{n+1}\mbox{ of type
}h\}.
$$
Clearly, as $h$ varies, the sets $S_n^h(I,\lambda)$ may have nonempty
intersection but the important point is that
$$
\dim S_n(I,\lambda) = \max_{0\le h\le \lfloor \frac{n}{2}\rfloor}
\dim S_n^h(I,\lambda).
$$

In~\cite[Section 4]{bad_bug_2020} the authors show that\footnote{In
fact, the value $h$ in~\cite{bad_bug_2020} equals $h+1$ in this
paper, hence the formula looks slightly different.}  for $h< n/2$
\begin{equation}\label{eq25}
\dim S_n^h(I,\lambda) \le \frac{(h+2)(1-h\lambda)}{(n -
2h)(1+\lambda)}.
\end{equation}
Here we prove a different upper bound for $\dim S_n^h(I,\lambda)$:

\begin{proposition}\label{prop3}
For all $0\le k\le \min\{h,\frac{n-1}{2}\}$, one has
$$
\dim
S_n^h(I,\lambda)\le\frac{(n-k+1)(1-k\lambda)}{(n-2k)(1+\lambda)}.
$$
\end{proposition}
\proof We proceed in a similar way as in Theorem~\ref{th4}. Fix $k$
between 0 and $\min\{h,\frac{n-1}{2}\}$. For a given point
$\vq\in\ZZ^{n+1}$ denote by $\vq_i^{n-k}$ the vector $(q_i,q_{i+1},
\ldots, q_{i+n-k})$. Consider a point $\vq\in\ZZ^{n+1}$ of type $h$
that approximates the number $x\in S_n^h(I,\lambda)$. Let
$\LLL_\vq^k$ be the lattice generated by vectors $\vq_0^{n-k},
\ldots, \vq_k^{n-k}$. By the definition of the type of $\vq$, all
these vectors are linearly independent and therefore the covolume of
$\LLL_\vq^k$ is $||\vq_0^{n-k}\wedge \cdots\wedge \vq_k^{n-k}||_2$.
By~\cite[Proposition 4.3]{bad_bug_2020}, the absolute value of every
Pl\"ucker coordinate of this product is bounded from above by
$q_0^{1-k\lambda}$ multiplied by a some absolute constant.

Since the covolumes of a lattice and its dual counterparts coincide,
we get
$$
\covol (\LLL_\vq^{h\bot}) \ll q_0^{1-k\lambda}.
$$
Its dimension is $n-2k$. By Minkowski's Theorem, this implies that
there exists an integer vector $\va\in\langle \vq_0^{n-k}, \ldots,
\vq_k^{n-k}\rangle^\bot$ of length
$$
||\va||_\infty \ll q_0^{\frac{1-k\lambda}{n-2k}},
$$
i.e. $\vq\in L_\va$.

Now, as in the proof of Theorem~\ref{th4}, we have
$|q_0|\asymp||\vq||_\infty$ and compute
$$
|\va\cdot (1,x,x^2,\ldots, x^{n-k})| \ll
\frac{||\va||_\infty}{||\vq||_\infty^{1+\lambda}} \ll
||\va||_\infty^{1 - \frac{(n-2k)(1+\lambda)}{1-k\lambda}}.
$$
Finally, we derive that
$$
S_n^h(I,\lambda)\subset D_{n-k}\left(I,
\frac{(n-2k)(1+\lambda)}{1-k\lambda} -1\right)
$$
and Bernik's equation~\eqref{bernik} completes the proof.

\endproof

Notice that for all $h<k\le \frac{n-1}{2}$ one has
$$
h+2\le \frac{n+3}{2}\le n-k+1.
$$
On top of that, for $\lambda\le \frac{2}{n}$, the expression
$\frac{1-k\lambda}{n-2k}$ as a function of $k$, is monotonically increasing.
Therefore,
$$
\frac{(h+2)(1-h\lambda)}{(n-2h)(1+\lambda)}\le \min_{h<k\le
\frac{n-1}{2}}\left\{\frac{(n-k+1)(1-k\lambda)}{(n-2k)(1+\lambda)}\right\},
$$
and hence in view of~\eqref{eq25}, $\dim S_n^h(I,\lambda)$ satisfies
the same upper bounds~\eqref{th1_eq} for all $0\le k\le
\frac{n-1}{2}$. This finishes the proof of Theorem~\ref{th1}.

\section{Cutting the curve into pieces}\label{sec3}

From now on we focus on Theorem~\ref{th2}. By $Q^\vf_n(I, \lambda)$ we denote
the set of all $\vq\in\ZZ^{n+1}$ such that there exists $x\in I$ with
\begin{equation}\label{approx}
\max_{1\le i\le n} |q_0 f_i(x) - q_i| < ||\vq||_\infty^{-\lambda}.
\end{equation}
Sometimes it is convenient to write $\vq$ as a pair $(q_0,
\vq^+)\in\ZZ\times\ZZ^n$. Then~\eqref{approx} can be rewritten as
$||q_0\vf(x) - \vq^+||_\infty \ll ||\vq||_\infty^{-\lambda}$. By
$R^\vf(\vq)$ we denote the set of $x\in\RR$ that
satisfy~\eqref{approx} for a given $\vq\in\ZZ^{n+1}$. Surely, one
has $\diam R^\vf(\vq)\ll ||\vq||_\infty^{-1-\lambda}$. Notice that
$S^\vf_n(I,\lambda)$ can be interpreted as the set of all $x\in \RR$
such that the inequality~\eqref{approx} is satisfied for infinitely
many $\vq\in Q^\vf_n(I,\lambda)$. As with $S_n(I,\lambda)$, we omit
the superscripts in $Q_n(I,\lambda)$ and $R(\vq)$ when $\CCC$ is the
Veronese curve.

We split $Q^\vf_n(I,\lambda)$ into subsets $Q^\vf_n(I,\lambda,k)$ where
$k\in\NN$ and
$$
Q^\vf_n(I,\lambda,k):=\{\vq\in Q^\vf_n(I,\lambda): 2^k \le ||\vq||_\infty
<2^{k+1}\}.
$$
We also define
$$
S^\vf_n(I,\lambda, k):= \bigcup_{\vq\in Q^\vf_n(I,\lambda, k)} R^\vf(\vq).
$$
Notice that it is possible for $S_n^\vf(I,\lambda,k)$ to contain
points outside $I$. On the other hand, every term $R^\vf(\vq)$ in
the union $S_n^\vf(I,\lambda,k)$ has a non-empty intersection with
$I$. Since their diameters tend to zero as $k\to\infty$, any
$x\not\in I$ can belong to at most finitely many sets
$S_n^\vf(I,\lambda,k)$. Because of that and since each
$Q^\vf_n(I,\lambda, k)$ contains finitely many elements, we can
write $S^\vf_n(I,\lambda)$ as a limsup set:
$$
S^\vf_n(I,\lambda) = \limsup_{k\to\infty} S^\vf_n(I,\lambda,k).
$$

We further split the sets $Q_n^\vf(I,\lambda,k)$ into smaller
subsets. But before doing that, we need the following

\begin{lemma}\label{lem1}
Let $Q\in\RR^+$, $\vq\in\RR^{n+1}$ and $x\in I$ be such that $Q\le
||\vq||_\infty <2Q$ and
$$
||q_0 \vf(x) - \vq^+||_\infty\le ||\vq||_\infty^{-\lambda}.
$$
Then for all $x_0\in I\cap B(x,Q^{-\frac{1+\lambda}{2}})$ one has
\begin{equation}\label{lem1_eq}
\left\{\begin{array}{l} |q_0|\ll Q\\
|q_0x_0 - q_1|\ll Q^{\frac{1-\lambda}{2}};\\
|q_0(f_i(x_0)-x_0f_i'(x_0)) + q_1f_i'(x_0) - q_i| \ll
Q^{-\lambda},\qquad 2\le i\le n.
\end{array}
\right.
\end{equation}
\end{lemma}

\proof The first two inequalities immediately follow from the
relations between $\vq$ and $Q$ and the inclusion $x_0\in
B(x,Q^{-\frac{1+\lambda}{2}})$. For the last set of inequalities we
compute
$$
Q^{-\lambda}\gg |q_0f_i(x) - q_i| = \left|q_0\left(f_i(x_0) + (x -
x_0)f_i'(x_0) + \frac12 (x-x_0)^2 f_i''(\xi)\right) - q_i\right|
$$
$$
\ge |q_0(f_i(x_0) - x_0 f_i'(x_0)) + q_1f_i'(x_0) - q_i| -
|(q_1-q_0x)f_i'(x_0)| - \left|\frac12 q_0(x-x_0)^2
f_i''(\xi)\right|.
$$
One can easily check that the last two terms are $\ll Q^{-\lambda}$
and hence verify \eqref{lem1_eq}.
\endproof

Given $x_0\in I$ and $Q\in \RR^+$, we denote the box of all
$\vq\in\RR^{n+1}$ that satisfy~\eqref{lem1_eq} by $\Delta(x_0,Q)$.
An easy computation reveals
$$
\vol(\Delta(x_0, Q)) \asymp Q^{1+\frac{1-\lambda}{2} - (n-1)\lambda} = Q^{\frac{3-(2n-1)\lambda}{2}}.
$$
Therefore for $\lambda< \frac{3}{2n-1}$ this volume tends to
infinity as $Q\to\infty$.

By Lemma~\ref{lem1}, a neighbourhood of $\vf(x)$, $x\in I$ can be
covered by the boxes $\Delta(x_1, 2^k), \ldots$, $\Delta(x_d,2^k)$
where $d\ll 2^{\frac{1+\lambda}{2}k}$ such that $Q^\vf_n(I,\lambda,
k)$ is contained in the union of sets $Q^\vf_n(I, \lambda, k, m)$,
$1\le m\le d$ where
\begin{equation}\label{def_m}
Q^\vf_n(I, \lambda, k, m) := \Delta(x_m, 2^k)\cap \ZZ^{n+1}.
\end{equation}

Let $\tau_1(m)$, $\tau_2(m), \ldots,$ $\tau_{n+1}(m)$ be the successive
minima of $\Delta_m:=\Delta(x_m, 2^k)$ on $\ZZ^{n+1}$. By Minkowski's second
theorem, we know that
\begin{equation}\label{eq26}
\tau_1(m)\tau_2(m)\cdots \tau_{n+1}(m)\asymp
\frac{\vol(\RR^{n+1}/\ZZ^{n+1})}{\vol \Delta_m} \asymp
2^{\frac{(2n-1)\lambda - 3}{2}k}.
\end{equation}
We also know that if $\tau_{n+1}(m) < 1$, i.e. $\Delta_m$ contains
$n+1$ linearly independent integer vectors, then
\begin{equation}\label{eq2}
\#Q^\vf_n(I, \lambda, k, m) = \#(\Delta_m\cap \ZZ^{n+1})\ll
2^{\frac{3-(2n-1)\lambda}{2}k}.
\end{equation}
By $\#^*Q^\vf_n(I, \lambda, k, m)$ we denote the number of primitive
points in $Q^\vf_n(I, \lambda, k, m)$, i.e. points $\vq\in\ZZ^{n+1}$
such that $q_0>0$ and $\gcd(q_0, q_1,\ldots, q_n)=1$. We split the
sets $Q^\vf_n(I, \lambda, k, m)$ into two types: type~1 ones satisfy
$$
\#^*Q^\vf_n(I, \lambda, k, m)\ll 2^{\frac{3-(2n-1)\lambda}{2}k},
$$
where the implied absolute constant is the same as in the
bound~\eqref{eq2}. The other sets are called type~2. That guarantees
that for $Q_n^\vf(I,\lambda, k,m)$ of type~2 the corresponding
minimum $\tau_{n+1}(m)>1$. Then we define
$$
Q^\vf_{n,1}(I, \lambda, k):= \bigcup_{Q_n^\vf
(I,\lambda,k,m)\mathrm{\; is\; of\; type\;}1} Q_n^\vf(I,\lambda,k,m)
$$
and
$$
Q_{n,2}^\vf (I,\lambda,k):= Q_n^\vf(I,\lambda,k)\setminus
Q_{n,1}^\vf (I,\lambda,k).
$$
%

Respectively, we split $S^\vf_n(I,\lambda)$ into the subsets
$S^\vf_{n,i}(I,\lambda)$, $i\in\{1,2\}$ where
$$
S^\vf_{n,i}(I,\lambda) = \limsup_{k\to\infty} S^\vf_{n,i}(I,\lambda, k):=
\limsup_{k\to\infty} \bigcup_{\vq\in Q^\vf_{n,i}(I,\lambda, k)} R^\vf(\vq).
$$

\begin{lemma}\label{lem2}
One has
$$
\dim S^\vf_{n,1}(I,\lambda)\le \frac{2-(n-1)\lambda}{1+\lambda}.
$$
\end{lemma}

\proof We consider the natural cover of the set
$S^\vf_{n,1}(I,\lambda)$ by $R^\vf(\vq)$ where $\vq\in
Q^\vf_{n,1}(I,\lambda, k)$ for $k\ge K_0$ sufficiently large. In
view of $\diam R^\vf(\vq)\ll ||\vq||_\infty ^{-1-\lambda}$, we get
that the value of the corresponding Hausdorff $s$-sum tends to zero
as $K_0\to\infty$ as soon as the series
$$
\sum_{k=1}^\infty \sum_{\vq\in Q^\vf_{n,1}(I,\lambda,k)\atop
\vq\;\mathrm{is\; primitive}} 2^{-(1+\lambda)sk}
$$
converges.

Notice that the number of boxes $\Delta_j$ (and in turn sets
$Q_n^\vf(I,\lambda,k,j)$) that correspond to each
$Q^\vf_{n,1}(I,\lambda,k)$ is bounded from above by $\ll
2^{\frac{1+\lambda}{2}k}$. Also each $\vq\in Q_{n,1}^\vf (I,
\lambda,k)$ has to belong to one of those $Q_n^{\vf}(I,\lambda,k,j)$
of type~1 that contain $\ll 2^{\frac{3-(2n-1)\lambda}{2}k}$
primitive points. In total, we get the series
$$
\sum_{k=1}^\infty 2^{\left(\frac{1+\lambda}{2} +
\frac{3-(2n-1)\lambda}{2} - (1+\lambda)s\right)k}.
$$
This series converges as soon as
$$
\frac{1+\lambda}{2} + \frac{3-(2n-1)\lambda}{2} - (1+\lambda)s <
0\quad \Leftrightarrow\quad s > \frac{2-(n-1)\lambda}{1+\lambda}.
$$
\endproof

Now we focus on the sets $Q^\vf_{n,2}(I,\lambda, k)$. Consider the
set $Q^\vf_n(I,\lambda, k,m)$ of type~2. By~\eqref{eq2}, we then
must have $\tau_{n+1}(m)>1$, i.e. this set belongs to a proper
subspace of $\RR^{n+1}$. By the height of a rational subspace $\SSS$
we define the volume of the fundamental domain of the lattice
$\LLL:=\SSS\cap \ZZ^{n+1}$. In other words, if the lattice $\LLL$ is
generated by $\vv_1, \vv_2, \ldots, \vv_d$ then $H(\SSS):=
||\vv_1\wedge \vv_2\wedge\ldots \wedge \vv_d||_2$. We refer the
reader to~\cite[Chapter 1, \S 5]{schmidt_1991} for more details
about this notion.

\begin{lemma}\label{lem3}
Let $Q^\vf_n(I,\lambda, k,m)$ be of type two. Suppose that the dimension of
$\Span (Q^\vf_n(I,\lambda, k,m))$ is $d$. Then
$$
H(\Span (Q^\vf_n(I,\lambda, k,m)))\ll 2^{(n-d+1)\lambda k}.
$$
Moreover, if $\tau_{n+1}(m) = 2^{\delta k}$ then $\Span (Q^\vf_n(I,\lambda,
k,m))$ belongs to an $n$-dimensional rational subspace of height $\ll
2^{(\lambda - \delta)k}$.
\end{lemma}

\proof Since $Q^\vf_n(I,\lambda, k,m)$ is of type two, we must have
at least one of the successive minima $\tau_i(m)$ bigger than 1.
Suppose $d$ of them are at most 1 ($d<n+1$), i.e. $\tau_1(m)\le
\cdots\le \tau_d(m)\le 1< \tau_{d+1}(m)\le \cdots \le
\tau_{n+1}(m)$. Obviously, $d>1$ because otherwise all integer
points $\vq\in Q^\vf_n(I,\lambda, k,m)$ are scalar multiples of one
point $\vq_0$ and therefore $Q^\vf_n(I,\lambda, k,m)$ is not of
type~2. Denote by $\vv_i$ the shortest vector that corresponds to
the successive minimum $\tau_i(m)$. Then $Q^\vf_n(I,\lambda, k,m)$
lies in a $d$-dimensional rational subspace which contains the
lattice with generators $\vv_1, \ldots, \vv_d$.

From Minkowski's second theorem, similarly to~\eqref{eq26}, we get
that $\tau_1(m)\cdots \tau_d(m)\ll 2^{\frac{(2n-1)\lambda -
3}{2}k}$. We also have
$$
\vol(\Span (Q^\vf_n(I,\lambda, k,m))\cap \Delta_m) \ll 2^k\cdot
2^{\frac{1-\lambda}{2}k}\cdot 2^{-(d-2)\lambda k} =
2^{\frac{3-(2d-3)\lambda}{2}k}.
$$
Therefore
$$
\left|\left|\bigwedge_{i=1}^d \vv_i\right|\right|_2\ll \prod_{i=1}^d
\tau_i(m)\cdot  \vol(\Span (Q^\vf_n(I,\lambda, k,m))\cap \Delta_m)
\ll2^{(n-d+1)\lambda k}.
$$

To get the second statement of the lemma, we consider the span of $\vv_1,
\ldots, \vv_n$ and notice that
$$
\tau_1(m)\cdots \tau_n(m)\asymp 2^{\left(\frac{(2n-1)\lambda - 3}{2}
- \delta\right)k}.
$$
Then proceeding as above, we get $||\vv_1\wedge\ldots\wedge
\vv_n||_2 \ll 2^{(\lambda - \delta)k}$. Finally, the observation
that the $\Span (Q^\vf_n(I,\lambda, k,m))$ belongs to the span of
$\vv_1, \ldots, \vv_n$ finishes the proof. \endproof

Given the box $\Delta_m$, denote by $\delta(m)$ the value that satisfies
$$\tau_{n+1}(m) =: 2^{\delta(m) k}.$$ Note that for the boxes
$Q_n^\vf(I,\lambda,k,m)$ of type two, $\delta(m)>0$.

\begin{lemma}\label{lem4}
The number of points in $\Delta_m\cap \ZZ^{n+1}$ satisfies
$$
\#(\Delta_m\cap \ZZ^{n+1}) \ll 2^{\left(\frac{3 - (2n-1)\lambda}{2}
+ (n-1)\delta(m)\right)k}
$$
\end{lemma}

\proof We estimate the number of integer points in $\Delta_m\cap
\ZZ^{n+1}$ as
$$
\# (\Delta_m\cap \ZZ^{n+1}) \asymp \prod_{\tau_i(m)<1}
\tau_i^{-1}(m) \stackrel{\eqref{eq26}}\asymp
2^{\frac{3-(2n-1)\lambda}{2}k} \prod_{\tau_i(m)\ge 1} \tau_i(m).
$$
To finish the proof, we notice that there are at most $n-1$ successive minima
$\tau_i(m)$ which are at least 1 and that for all $1\le i\le n+1$,
$\tau_i(m)\le \tau_{n+1}(m)$. \endproof

Denote by $\PPP(m)$ a hyperplane of the smallest possible height
that contains all integer points from $\Delta_m$ and denote its
equation by $\va(m)\cdot \vx = 0$. Let $\sigma(m)$ be such that
$$H(\PPP(m)) =: 2^{(\lambda - \sigma(m))k}.$$ Lemma~3 tells us that
$\sigma(m)\ge \delta(m)$, however in some cases it can be made
substantially bigger. For a fixed $\sigma>0$, denote by
$Q_{n,2}^{\vf, \sigma} (I,\lambda, k)$ the set of $\vq\in\ZZ^{n+1}$
from all the boxes $Q^\vf_n(I,\lambda,k,m)$ of type~2 such that
$\sigma(m)\ge \sigma$, i.e.
$$
Q_{n,2}^{\vf,\sigma}(I,\lambda, k):= \bigcup_{\sigma(m)\ge
\sigma\atop Q_n^\vf(I,\lambda,k,m) \mathrm{\;is\; of\; type\; 2}}
Q_{n}^\vf (I,\lambda,k,m)
$$

 Also define
$$
S_{n,2}^{\vf, \sigma}(I,\lambda) = \limsup_{k\to\infty}
S_{n,2}^{\vf, \sigma}(I,\lambda, k):= \limsup_{k\to\infty} \bigcup_{\vq\in
Q_{n,2}^{\vf, \sigma}(I,\lambda, k)} R^\vf(\vq).
$$

\begin{proposition}\label{prop1}
Suppose that for any rational hyperplane $\HHH\subset \RR^{n+1}$,
$\dim ((1,\vf(I))\cap \HHH)=0$. Then
$$
\dim S_{n,2}^{\vf, \sigma}(I,\lambda) \le \dim D^\vf_n\left(I,
\frac{1+\lambda}{\lambda - \sigma}-1\right).
$$
\end{proposition}

\proof We proceed as in the proof of Theorem~\ref{th1}. By
definition, for fixed $m$ such that $Q^\vf_n(I,\lambda,k,m) \subset
Q_{n,2}^{\vf,\sigma}(I,\lambda,k)$ all points $\vq\in
Q^\vf_n(I,\lambda,k,m)$ lie on a hyperplane $\va\cdot \vx=0$ with
$||\va||_\infty \ll 2^{(\lambda-\sigma)k}\asymp
||\vq||_\infty^{\lambda-\sigma}$. Then by analogy with the
inequality~\eqref{eq1}, for any $\vq\in Q^\vf_n(I,\lambda, k,m)$ and
any $x\in R^\vf(\vq)$ we get
\begin{equation}\label{eq3}
|\va\cdot (1,\vf(x))|\ll ||\va||_\infty ^{1 - \frac{1+\lambda}{\lambda
- \sigma}}.
\end{equation}

Now consider $x\in S_{n,2}^{\vf, \sigma}(I,\lambda)$. We have shown that such
$x$ must satisfy~\eqref{eq3} infinitely often. By removing the set of the
Hausdorff dimension zero, we make sure that~\eqref{eq3} is satisfied for
infinitely many distinct vectors $\va$. That implies that $S_{n,2}^{\vf,
\sigma} (I,\lambda) \subset D^\vf_n\left(I,\frac{1+\lambda}{\lambda -
\sigma}-1\right)$ and the proposition follows immediately. \endproof

Set $\sigma_0 = \frac{2(n\lambda-1)}{n+1}$. Assuming that $\dim
D^\vf_n(I, w) \le \frac{n+1}{w+1}$, which is the case for the
Veronese curve, Proposition~\ref{prop1} implies
$$
\dim S_{n,2}^{\vf, \sigma_0} (I,\lambda) \le
\frac{2-(n-1)\lambda}{1+\lambda}.
$$
We say that the box $Q^\vf_n(I,\lambda, k,m)$ is of type~3 if it is
of type~2 and satisfies $\sigma(m)< \sigma_0$. Analogously as for
$S_{n,2}^{\vf,\sigma}(I,\lambda)$, we define the set
$S_{n,3}^\vf(I,\lambda)$ as the set of $x\in I$ that lie in
$R^\vf(\vq)$ for infinitely many $\vq$ from the boxes of type~3.
that is
$$
S_{n,3}^\vf(I,\lambda) = \limsup_{k\to\infty} S_{n,3}^{\vf,
\sigma}(I,\lambda, k):= \limsup_{k\to\infty} \bigcup_{\vq\in
Q_{n,2}^{\vf}(I,\lambda, k)\setminus
Q_{n,2}^{\vf,\sigma_0}(I,\lambda,k)} R^\vf(\vq).
$$

For a given box $\Delta_m$ consider $\eta\ge 0$ such that there
exists a smaller box inside $\Delta_m$ defined by the inequalities
$$
\Delta^*(x,2^k):=\left\{ \vq\in \RR^{n+1}: \begin{array}{l} |q_0|\ll
2^k;\\
|q_0x - q_1|\ll 2^{\left(\frac{1-\lambda}{2} - \eta\right)k};\\
|q_0(f_i(x) - xf'_i(x)) + q_1f_i'(x) - q_i| \ll 2^{-\lambda k}
\end{array}\right\}
$$
which contains all integer points of $\Delta_m$. By $\eta^*(m)$ we define the
supremum of $\eta$ with this property. Here the implied constants in the
inequalities are the same as in~\eqref{lem1_eq}. Then the same computations
as in Lemma~\ref{lem3} imply
\begin{equation}\label{eq4}
||\va||_\infty \ll 2^{(\lambda - \eta^*(m) - \delta(m))k}.
\end{equation}
In particular, for boxes of type~3 we must have
$\eta^*(m)+\delta(m)\le \sigma_0$.

Before stating the next lemma, recall that the centres $x_m$ of
$\Delta_m$ are defined together with sets $Q_n^\vf(I,\lambda,k,m)$
in~\eqref{def_m}.

\begin{lemma}\label{lem6}
Suppose that $Q^\vf_n(I,\lambda, k,m)$ is of type~3. Then there exists an
interval $J = J(m)\subset B(x_m,2^{-\frac{1+\lambda}{2}k})$ such that
\begin{equation}\label{lem6_eq1}
|J| \asymp 2^{\left(-\frac{1+\lambda}{2}-\eta^*(m)\right)k},
\end{equation}
and that for all $x\in J$,
\begin{equation}\label{lem6_eq2}
|\va(m)\cdot (1,\vf(x))|\ll ||\va(m)||_\infty 2^{-(1+\lambda)k}.
\end{equation}
Moreover, for all $\vq\in Q^\vf_n(I,\lambda, k,m)$ we have
$R^\vf(\vq)\subset J$.
\end{lemma}

This lemma tells us that for boxes $Q^\vf_n(I,\lambda, k,m)$ of type~3 we
must have a linear form which takes very small values on unusually long
intervals. Such a phenomenon is usually quite rare so the next idea will be
to provide an upper bound for the number of its occurrences.


\proof By the definition of type~2 (and hence type~3) box, there
exists $\vq\in \PPP(m)\cap \Delta_m$ with $2^k\ge q_0\ge 2^{k-1}$.
Also by definition of $\eta^*(m)$, taking $x = q_1/q_0$, there exist
two points $\pm \vu\in \PPP(m)\cap \Delta_m$ such that $|u_0x - u_1|
\asymp 2^{\left(\frac{1-\lambda}{2} - \eta^*(m)\right)k}$. Since
$\PPP(m)\cap \Delta_m$ is convex, the whole triangle with vertices
$-\vu,\vu,\vq$ belongs to $\PPP(m)\cap \Delta_m$. Next, at least one
of the midpoints $\vp$ of the segments $\vu,\vq$ and $-\vu,\vq$
satisfy $p_0 \ge 2^{k-2}$, hence $p_0\asymp 2^k$. Fix that midpoint
$\vp$ and consider the intersection of the triangle $-\vu, \vu,\vq$
with the hyperplane given by the equation $x_0 = p_0$ and call the
resulting segment $S$. Let $T$ be the triangle with one of the sides
$S$ and the opposite vertex $\vq$.

Consider
$$J:= \left\{\frac{x_1}{x_0} : \vx\in T\right\}.$$
Since $T$ is a connected set, $J$ is an interval. By construction,
we have $x\in J$, $p_1/p_0\in J$ and $|p_0x - p_1| \asymp
2^{\left(\frac{1-\lambda}{2} - \eta^*(m)\right)k}$. Since $p_0\asymp
2^k$, the length of $J$ satisfies~\eqref{lem6_eq1}. It is also not
hard to check that all $y = x_1/x_0\in J$ satisfy $|x_0f_i(y) - x_i|
\ll 2^{-\lambda k}$ for all $1\le i\le n$. Finally, by construction
of $T$, $\vx\in T$ implies $\va(m)\cdot \vx=0$, therefore
$$
|x_0 (\va\cdot (1,\vf(y)))| = |a_1(x_0f_1(y)-x_1) +
a_2(x_0f_2(y)-x_2) + \ldots + a_n(x_0f_n(y) - x_n)| \ll
||\va||_\infty 2^{-\lambda k},
$$
where $\va = \va(m) = (a_0, a_1, \ldots, a_n)$. Dividing by $x_0\gg 2^k$
gives~\eqref{lem6_eq2}.

If needed, we extend the interval $J$ so that it contains all
neighbourhoods $R^\vf(\vq)$ for $\vq\in Q^\vf_n(I,\lambda,k,m)$. By
the definition of $\eta^*(m)$, we have $|J|\ll
2^{(-\frac{1+\lambda}{2}-\eta^*(m))k}$ and hence $|J|$ is bounded
from above by~\eqref{lem6_eq1} as well.
\endproof

Denote by $J^\vf(m)$ the largest interval that contains $J(m)$ from
Lemma~\ref{lem6} such that for all $x\in J^\vf(m)$, the
inequality~\eqref{lem6_eq2} is satisfied with the same implied
constant. Define $\eta(m)$ from the equation
$$
|J^\vf(m)| =: 2^{\left(-\frac{1+\lambda}{2} -\eta(m)\right) k}.
$$
From Lemma~\ref{lem6} and construction of boxes of type~3, we get that
\begin{equation}\label{eq32}
\eta(m) \le \eta^*(m) \le \sigma_0-\delta(m) =
\frac{2(n\lambda-1)}{n+1} - \delta(m) < \frac{\lambda + 1}{2}.
\end{equation}
 On the other hand,
$\eta(m)$ may be negative. In that case we have several consecutive
boxes $Q_n^\vf (I,\lambda, k,m)$ that share the same hyperplane
$\PPP(m)$. The number of such boxes is bounded from above by $\ll
2^{-\eta(m)k}$.

Let $\epsilon>0$ be an arbitrarily small but fixed value. Consider
all intervals $J^\vf(m)$ that correspond to boxes
$Q_n^\vf(I,\lambda, k,m)$ of type~3. Group these intervals into
families $L_n^{\vf,\eta} (I,\lambda, k)$, where $\eta$ takes the
values $-M\epsilon, -(M-1)\epsilon, \ldots, 0,\epsilon, \ldots,
N\epsilon$, such that $M,N\in\ZZ$ will be defined below and for
every interval $J^\vf(m)$ in $L_n^{\vf,\eta} (I,\lambda, k)$, its
corresponding value $\eta(m)$ satisfies $\eta-\epsilon<\eta(m)\le
\eta$, i.e.
$$
L_n^{\vf,\eta}(I,\lambda,k):= \bigcup_{\eta-\epsilon<\eta(m)\le
\eta\atop Q_n^\vf(I,\lambda,k,m) \mathrm{\;is\; of\; type\; 3}}
J^\vf(m).
$$

 Then we can write the set $S_{n,3}^\vf(I,\lambda)$ as
the union of the subsets $S_{n,3}^{\vf,\eta}(I,\lambda)$ where each
$x\in S_{n,3}^{\vf,\eta}(I,\lambda)$ belongs to intervals from
$L_n^{\vf,\eta}(I,\lambda, k)$ for infinitely many values $k$.

Notice that $\eta(m)\le \sigma_0$ therefore we can take $N =
\sigma_0\epsilon^{-1}$. On the other hand, we surely have $|J(m)|\le
|I|\asymp 1$ therefore $\eta(m) \ge  - \frac{1+\lambda}{2}$ and we
can take $M = \frac{1+\lambda}{2\epsilon}$. We conclude that the
number of families $L_n^{\vf,\eta} (I,\lambda, k)$ is bounded from
above by a value which is independent of $k$.

Let $c_1,c_2,c_3>0$ be some arbitrary but fixed constants which may
only depend on $I$ and $\vf$ but not on $k$. Denote by $A^{\vf,
\eta}_n(I,\lambda, k)$ the set of integer vectors $\va\in\ZZ^{n+1}$
such that
\begin{equation}\label{def_ane}
||\va||_\infty \le \left\{\begin{array}{rl}c_12^{(\lambda -
\eta+\epsilon)k}& \mbox{if } \eta>0\\[1ex]
c_12^{\lambda k}&\mbox{ if }\eta\le 0
\end{array}\right.
\end{equation}
and there exists an interval $J\subset I$ that satisfies the
following conditions:
\begin{itemize}
\item  For all $x\in J$,
$|\va\cdot(1,\vf(x))|\le c_2||\va||_\infty 2^{-(1+\lambda)k}$;
\item $|J|\ge c_32^{-\left(\frac{1+\lambda}{2} +
\eta\right)k}$;
\item $J$ can not be extended to the interval of length $c_32^{-\left(\frac{1+\lambda}{2} +
\eta-\epsilon\right)k}$ so that the first property from this list is
satisfied.
\end{itemize}
For a given $\va\in A_n^{\vf,\eta}(I,\lambda,k)$, we denote the
corresponding interval by $J^{\vf,\eta}(\va,k)$. Next, define
$$J^{\vf,\eta}_n(I,\lambda,k):=\{J^{\vf,\eta}(\va,k): \va\in
A_n^{\vf,\eta}(I,\lambda,k)\}.$$ Finally, denote by
$D^{\vf,\eta}_n(I,\lambda)$ the following limsup set:
$$
D^{\vf,\eta}_n(I,\lambda) = \limsup_{k\to\infty}
D^{\vf,\eta}_n(I,\lambda, k):= \limsup_{k\to\infty} \bigcup_{J\in
J^{\vf,\eta}_n(I,\lambda,k)} J.
$$
If $c_1,c_2,c_3$ in the inequalities for $||\va||_\infty$,
$|\va\cdot(1,\vf(x))|$ and $|J|$ are the same as the implied constants
in~\eqref{eq4},~\eqref{lem6_eq1},~\eqref{lem6_eq2} respectively then
Lemma~\ref{lem6} implies that $S_{n,3}^{\vf,\eta}(I,\lambda) \subset
D_n^{\vf,\eta}(I,\lambda)$. In the coming sections we will compute an upper
bound for $\#A_n^{\vf,\eta}(I,\lambda,k)$. In particular, it will provide an
upper bound for the Hausdorff dimension of $D_n^{\vf,\eta}(I,\lambda)$. That
bound will be the same for any triple $c_1,c_2,c_3>0$. Hence that will give
us an upper bound for $\dim S_{n,3}^{\vf,\eta}(I,\lambda)$.

When proving Theorem~\ref{th2} for $n=3$, we further group the families
$L_n^{\vf,\eta}(I,\lambda,k)$ into subfamilies
$L_n^{\vf,\eta,\delta}(I,\lambda,k)$, where $\delta$ takes values
$0,\epsilon, \ldots, K\epsilon$, such that the corresponding interval
$J^\vf(m)$ in $L_n^{\vf,\eta,\delta}(I,\lambda,k)$ satisfies $\delta<
\delta(m) \le \delta+\epsilon$. If several values of $m$ share the same
interval $J^\vf(m)$ (that is the case when $\eta(m)<0$), we ask that the
maximum of all correspondent values $\delta(m)$ lies between $\delta$ and
$\delta+\epsilon$. Since $\delta(m)\le \sigma_0$ we can take $K =
\sigma_0\epsilon^{-1}$. We also write the set $S_{n,3}^\vf(I,\lambda)$ as the
union of subsets $S_{n,3}^{\vf,\eta,\delta}(I,\lambda)$ in the same way as
before. Denote
\begin{equation}\label{def_aned}
A_n^{\vf,\eta,\delta}(I,\lambda,k):=\left\{\begin{array}{rl}
\{\va\in A_n^{\vf,\eta}(I,\lambda,k) : ||\va||_\infty \le c_1
2^{(\lambda-\eta-\delta+\epsilon)k} \}&\mbox{if }\eta\ge 0\\[1ex]
\{\va\in A_n^{\vf,\eta}(I,\lambda,k) : ||\va||_\infty \le c_1
2^{(\lambda-\delta+\epsilon)k} \}&\mbox{if }\eta< 0
\end{array}\right.
\end{equation}
Finally, for a given $J\in L_n^{\vf,\eta,\delta}(I,\lambda,k)$
Lemma~\ref{lem4} implies that the number of rational points from
$\bigcup_m Q_n^\vf (I,\lambda,k,m)$ such that $R(\vq)\in J$ is
bounded from above by
\begin{equation}\label{eq27}
\ll \left\{\begin{array}{ll} 2^{\left(\frac{3-(2n-1)\lambda}{2} +
(n-1)\delta + (n-1)\epsilon\right)k}&\mbox{if } \eta\ge0,\\
2^{\left(\frac{3-(2n-1)\lambda}{2} + (n-1)\delta - \eta +
n\epsilon\right)k}&\mbox{if } \eta<0.
\end{array}\right.
\end{equation}

\section{Preliminary results for polynomials}

Now we will focus on the Veronese curve, i.e. $\vf = (x,x^2,\ldots, x^n)$.
The behaviour of rational points and linear forms near it is much better
understood than for a generic curve. As was discussed before, for all the
notions associated with this curve we omit the superscript $\vf$.

We start by adapting the arguments of R. Baker~\cite[Lemma
4]{baker_1976} to show that without loss of generality one can
assume for all $\va\in A_n^\eta(I,\lambda, k)$ (respectively for all
$\va\in A_n^{\eta,\delta}(I,\lambda,k)$) that $||\va||_\infty =
a_n$.

Define
$$
A^{*,\eta}_n(I,\lambda,k):= \{\va\in A^\eta_n(I,\lambda,k) :
||\va||_\infty = a_n\}
$$
The set $A_n^{*,\eta,\delta}(I,\lambda,k)$ is defined similarly. Let
$J^{*,\eta}_n(I,\lambda,k)$ be the corresponding subset of
$J^{\eta}_n(I,\lambda,k)$ and $D^{*,\eta}_n(I,\lambda)$ be the
corresponding limsup set.

\begin{lemma}\label{lem12}
Suppose that for any interval $I$ such that $\dist(I,0)\gg 1$, and
any $c_1,c_2,c_3>0$ one has $\dim D^{*,\eta}_n(I,\lambda)\le d$.
Then $\dim D^{\eta}_n(I,\lambda)\le d$.

If there exists a function $a(\lambda, k, \eta)$ such that
$\#A_n^{*,\eta}(I,\lambda,k)\ll_{I,c_1,c_2,c_3} a(\lambda,k,\eta)$ then we
also have $\#A_n^{\eta}(I,\lambda,k) \ll_{I,c_1,c_2,c_3} a(\lambda,k,\eta)$.
Analogously, if there exists a function $a(\lambda,k,\eta,\delta)$ such that
$\#A_n^{*,\eta,\delta}(I,\lambda,k)\ll_{I,c_1,c_2,c_3}
a(\lambda,k,\eta,\delta)$ then $\#A_n^{\eta,\delta}(I,\lambda,k)
\ll_{I,c_1,c_2,c_3} a(\lambda,k,\eta,\delta)$.
\end{lemma}

\proof Consider $\va\in A^{\eta}_n(I,\lambda,k)$ and the
corresponding interval $J = J^{\eta}(\va,k)$. Then $\va\cdot
(1,\vf(x))$ is a polynomial $P_\va(x)$. By~\cite[Lemma
1]{baker_1966}, there exists $j\in\{0,\ldots, n\}$ such that
$|P_\va(j)| \asymp ||\va||_\infty$. Consider $Q(x) := P_\va(x+j)$.
Notice that $H(Q)\asymp ||\va||_\infty$ and $|Q(y)| \ll H(Q)\cdot
2^{-(1+\lambda)k}$ for all $y\in J - j:= \{x-j: x\in J\}$. On the
other hand, we have that $|Q(0)|\asymp H(Q)$ therefore $\dist(J-j,
0)\gg 1$ for $k$ large enough. Finally, notice that $|J-j| = |J|$.

Now consider $C = (n+1)^{n+3}2^n \asymp 1$ and let $R(x) = (Cx)^n
Q((Cx)^{-1})$. As shown in~\cite[Lemma 4]{baker_1976}, $H(R)$ equals
the leading coefficient of $R$. We also have $H(R)\asymp
||\va||_\infty$ and
$$
|R(z)|\ll H(R) 2^{-(1+\lambda)k}\;\mbox{ for all }\; z =
\frac{1}{Cy} \in \frac{1}{C(J-j)}:=\left\{\frac{1}{C(x-j)}: x\in
J\right\}.
$$
By construction, we also have that $\dist(1/(C(J-j)),0)\gg 1$ and
$|1/(C(J-j))| \asymp |J|$. We conclude that the coefficient vector $\vb$ of
$R(x)$ belongs to $A^{*,\eta}_n(1/(C(J-j)),\lambda,k)$ where the constants
$c_1^*, c_2^*, c_3^*$ arising from the definition of $A_n^{*,\eta}$ satisfy
$c_1^*\asymp c_1$, $c_2^*\asymp c_2$ and $c_3^*\asymp c_3$. From here we
immediately have
$$
\#A_n^{\eta}(I,\lambda,k)\ll \sum_{i=0}^n
\#A^{*,\eta}_n(1/(C(J-j)),\lambda,k) \ll_{I,c_1,c_2,c_3}
a(\lambda,k,\eta,\delta).
$$
The bound for $\#A_n^{\eta,\delta}(I,\lambda,k)$ is achieved
analogously.

Let $x\in D^\eta_n(I,\lambda)$. Then there exists $j\in\{0,\ldots,
n\}$ which corresponds to infinitely many vectors $\va\in
A_n^\eta(I,\lambda)$ such that $x\in J^{\eta}(\va,k)$. That in turn
implies that $1/(C(x-j)) \in D^{*,\eta}_n(1/(C(I-j)),\lambda)$.

Define $f_j(z):= (Cz)^{-1} + j$. We get that
$$
D_n^\eta(I,\lambda) \subset \bigcup_{j=0}^n
f_j(D^{*,\eta}_n(1/C(I-j),\lambda)).
$$
Then the statement of the lemma immediately follows.
\endproof

For the rest of the paper we will be dealing with sets
$A_n^{*,\eta}(I,\lambda,k), A_n^{*,\eta,\delta}(I,\lambda,k)$ and
$D_n^{*,\eta}(I,\lambda)$. However, for the sake of convenience we will omit
the stars in their notations. That is, we now state that for any $\va\in
A_n^\eta(I,\lambda)$ the product $\va\cdot (1,\vf(x))$ is a polynomial
$P_\va(x) = a_nx^n + \cdots + a_1x+a_0$ such that $H(P_\va) = a_n$. To
shorten the notation, we set $Q:= 2^k$. Then for all $\va\in
A_n^{\eta}(I,\lambda,k)$ we can write
$$
||\va||_\infty = a_n \ll \left\{
\begin{array}{ll}Q^{\lambda-\eta+\epsilon}&\mbox{if } \eta\ge 0,\\
Q^{\lambda}&\mbox{if }\eta<0,
\end{array}\right.
$$
\begin{equation}\label{eq13}
|P_\va(x)| \ll a_n Q^{-1-\lambda},\quad \forall\; x\in J(\va, k)
\end{equation}
and
\begin{equation}\label{eq14}
|J^{\eta}(\va, k)| \gg Q^{-\frac{1+\lambda}{2} - \eta}.
\end{equation}

Notice that $P_\va$ has exactly $n$ roots (counting multiplicities)
$x_1, x_2,\ldots, x_n$. Since the leading coefficient of $P_\va$ has
the largest absolute value, it is well known that all of them
satisfy $|x_i|\ll 1$. We will also use the fact that the
discriminant of the polynomial
$$
D(P) := a_n^{2n-2} \prod_{i\neq j} (x_i - x_j)^2
$$
is an integer number. For a given $x\in \RR$, we order the roots
$x_1, x_2, \ldots, x_n$ in such a way that
$$
|x-x_1|\le|x-x_2|\le \ldots\le |x-x_n|.
$$
We also denote $|x-x_i|=: Q^{-\kappa_i} = Q^{-\kappa_i(x)}$ and $|x_i-x_j| =:
Q^{-\mu_{i,j}}$. Notice that for $i<j$, $Q^{-\mu_{i,j}} \le 2Q^{-\kappa_j}$.
We will write $\alpha \lesssim \beta, \alpha\approx \beta$ and
$\alpha\gtrsim\beta$ if $Q^\alpha\ll Q^\beta$, $Q^\alpha\asymp q^\beta$ and
$Q^\alpha\gg Q^\beta$ respectively. Then with this notation we have
$\mu_{i,j}\gtrsim 0$, $\mu_{i,j}\gtrsim \kappa_j$.

\begin{lemma}\label{lem10}
Suppose that for a given polynomial $P_\va$ with $||\va||_\infty =
a_n$ there exists $w\in \RR$, $w>0$ and $\eta\in\RR$, $-\frac{w}{2}
< \eta < \frac{w}{2}$ and an interval $J$ of length $|J| \gg
Q^{-\frac{w}{2}-\eta}$ such that $\forall x\in J$, $P_\va(x) <
a_nQ^{-w}$. Then the discriminant of $P_\va$ satisfies
\begin{equation}\label{lem10_eq}
D(P_\va) \ll a_n^{2n-2} Q^{-w+2\eta}.
\end{equation}
\end{lemma}
\proof Fix a point $x_0\in J$ and consider any $x\in J$. We get
$$
P_\va(x) = P_\va(x_0) + (x-x_0)P'_\va(x_0) + \ldots +
\frac{1}{n!}(x-x_0)^n P_\va^{(n)}(x_0).
$$
Let $y_1, y_2, \ldots, y_{n+1}\in J$ be such that $y_2-y_1 = \ldots =
y_{n+1}-y_n$, $y_1$ and $y_{n+1}$ are the endpoints of $J$. That immediately
implies $|y_{i+1}-y_i|= |J|/n \asymp |J|$ for all $1\le i\le n$. Also denote
$b_i = \frac{P_\va^{(i)}(x_0)}{i!}$. Then the values $b_i$ are the solutions
of the following matrix equation
$$
\left(\begin{array}{ccccc} 1&y_1-x_0&(y_1-x_0)^2&\cdots&
(y_1-x_0)^n\\
1&y_2-x_0&(y_2-x_0)^2&\cdots&
(y_2-x_0)^n\\
\vdots&\vdots&\vdots&\ddots&\vdots\\
1&y_{n+1}-x_0&\cdots&\cdots&(y_{n+1}-x_0)^n
\end{array}\right) \left(\begin{array}{c}
b_0\\b_1\\\vdots\\b_{n}
\end{array}\right) = \left(\begin{array}{c}
P_\va(y_1)\\P_\va(y_2)\\\vdots\\P_\va(y_{n+1})
\end{array}\right).
$$
Notice that on the left hand side we have a Vandermonde matrix. Let's call it
$V$. Since $|y_j-y_i|\gg Q^{-\frac{w}{2} - \eta}$ for all $1\le i<j\le n+1$,
its determinant is
$$
|\det V| = \prod_{1\le i<j\le n+1} |y_i-y_j|\gg
\left(Q^{-\frac{w}{2}-\eta}\right)^{\frac{n(n+1)}{2}}.
$$
Then Cramer's rule gives for $2\le i\le n$
\begin{equation}\label{eq9}
|P_\va(x_0)| \asymp |b_0| \ll a_nQ^{-w};\; |P'_\va(x_0)| \asymp
|b_1| \ll a_nQ^{-\frac{w}{2}+\eta};\; |P^{(i)}_\va(x_0)|\ll
a_nQ^{-\frac{2-i}{2}w + i\eta}.
\end{equation}

Let $x_1,x_2,\ldots, x_n$ be the roots of $P_\va$ such that $|x_1-x_0|\le
|x_2-x_0|\le\cdots\le |x_n-x_0|$. Notice that for all $x$ in the segment
between $x_1$ and $x_0$  and all $2\le i\le n$ we get $|x-x_i|\ll |x_0-x_i|$.
Together with~\eqref{eq9}, that implies

$$
a_nQ^{-\frac{w}{2}+\eta}\gg|P'_\va(x_0)|\gg |P'_\va (x_1)| = a_n |(x_1-x_2)\cdots
(x_1-x_n)|.
$$
Therefore $\mu_{1,2}+\mu_{1,3}+\ldots+\mu_{1,n} \gtrsim
\frac{w}{2}-\eta$. Taking into account that for all other values
$\mu_{i,j}$ we have $\mu_{i,j}\gtrsim 0$, we end up with
$$
|D(P_\va)|= a_n^{2n-2} Q^{-\sum_{i\neq j} 2\mu_{i,j}}\ll a_n^{2n-2}
Q^{-w+2\eta}.
$$
\endproof

Applied to vectors $\va\in A_n^{\eta}(I, \lambda,k)$,
Lemma~\ref{lem10} immediately implies that

\begin{equation}\label{eq15}
D(P_\va)\ll a_n^{2n-2}Q^{-1-\lambda + 2\eta}\ll \left\{\begin{array}{rl}
a_n^{2n-2 - \frac{1+\lambda - 2\eta}{\lambda-\eta+\epsilon}}&\mbox{if }\; \eta>0,\\[1ex]
a_n^{2n-2-\frac{1+\lambda - 2\eta}{\lambda}}&\mbox{if }\; \eta\le0.
\end{array}\right.
\end{equation}

On top of that, since $\max\{|P_\va(x_0)|, |P'_\va(x_0)|, \ldots,
|P^{(n)}_\va(x_0)|\} \gg a_n$, from the inequalities~\eqref{eq9} we
derive that $\eta$ can not be too small, namely
\begin{equation}\label{eq31}
\frac{n-2}{2}(1+\lambda) + n\eta\gtrsim
0\qquad\Longleftrightarrow\qquad \eta\gtrsim
-\frac{(n-2)(1+\lambda)}{2n}.
\end{equation}

The next step is to find an upper bound for the cardinality
$\#A_n^\eta(I,\lambda,k)$. As the inequality~\eqref{eq15} suggests, it will
follow from an upper bound for the number of polynomials of bounded degree
and discriminant. It is well known that polynomials of a given degree $n$
come in equivalence classes where $P\approx Q$ if there exists a M\"obius
transform $\mu(x) = \frac{ax+b}{cx+d}$ with the determinant $ad-bc = \pm 1$
such that $Q(x) = (cx+d)^n P\circ\mu(x)$. It is also well known that all the
polynomials in the same equivalence class share the same discriminant.
Therefore one can first estimate the number of equivalence classes which
share a given discriminant and then the number of representatives of bounded
height in a given equivalence class and finally sum over discriminants up to
a given bound.

\section{Polynomials from the same equivalence class}

In this section we investigate how do the coefficients of equivalent
polynomials link with each other. In principle, to prove the main
result we only need a part of Proposition~\ref{prop4} but we find
the machinery of propagating rational points near the Veronese curve
interesting enough on its own. Therefore here we provide more
details than actually required.

Consider
$$
B = \begin{pmatrix}
a&b\\
c&d
\end{pmatrix} \in \mathrm{M}_{2,2}(\ZZ).
$$
We construct a map $\phi_n:M_{2,2}(\ZZ)\to M_{n+1,n+1}(\ZZ)$ in the
following way: $\phi_n(B) = A = (\alpha_{i,j})_{0\le i,j\le n}$
where the entry $\alpha_{i,j}$ equals the coefficient at $x^j$ of
the polynomial $(ax+b)^i(cx+d)^{n-i}$. The formula for $\alpha_{ij}$
is then
\begin{equation}\label{phi_aij}
\alpha_{ij} = \sum_{h=0}^n \left({i\atop h}\right)\left({n-i \atop
j-h}\right) a^hb^{i-h}c^{j-h}d^{n-i-j+h},
\end{equation}
where we set all binomial coefficients $\big({n\atop m}\big)$ equal
zero for $m<0$ or $m>n$. If one of the terms $a,b,c,d$ equals zero
then we still set the corresponding terms in the sum to be zero
where $a,b,c$ or $d$ is taken to the negative power.

The important property of this map is the following
\begin{proposition}\label{prop2}
$\phi_n$ is a monoid homomorphism from $M_{2,2}(\ZZ)$ to
$M_{n+1,n+1}(\ZZ)$. Restricted to $\SL_2(\ZZ)$, it is also a group
homomorphism from $\SL_2(\ZZ)$ to $\SL_{n+1}(\ZZ)$.
\end{proposition}

\proof Let
$$
B = \begin{pmatrix}
a&b\\
c&d
\end{pmatrix},\qquad C_1 = \begin{pmatrix}
1&0\\
1&1
\end{pmatrix},\qquad C_2= \begin{pmatrix}
0&1\\
-1&0
\end{pmatrix},\qquad C_3= \begin{pmatrix}
e&0\\
0&1
\end{pmatrix}.
$$

Consider the entry $(i,j)$ of $\phi_n(BC_1)$. By~\eqref{phi_aij}, it
equals
$$
\sum_{h=0}^n \left({i\atop h}\right)\left({n-i \atop j-h}\right)
(a+b)^hb^{i-h}(c+d)^{j-h}d^{n-i-j+h},
$$
$$
= \sum_{h=0}^n \left({i\atop h}\right)\left({n-i \atop j-h}\right)
\sum_{k=0}^n \left({h\atop k}\right) a^k b^{i-k} \sum_{r=0}^n
\left({j-h\atop r-k}\right)c^{r-k}d^{n-i-r+k}
$$
$$
= \sum_{k=0}^n a^kb^{i-k} \sum_{h=0}^n\sum_{r=0}^n  \left({i\atop
h}\right)\left({n-i \atop j-h}\right)\left({h\atop k}\right)
\left({j-h\atop r-k}\right)c^{r-k}d^{n-i-r+k}
$$
$$
= \sum_{k=0}^n \left({i\atop k}\right) a^k b^{i-k} \sum_{r=0}^n
\left({n-i\atop r-k}\right) \sum _{h=0}^n \left({i-k\atop
h-k}\right)\left({n-i-r+k\atop j-h-r+k}\right)c^{r-k}d^{n-i-r+k}
$$
$$
= \sum_{k=0}^n \left({i\atop k}\right) a^k b^{i-k} \sum_{r=0}^n
\left({n-i\atop r-k}\right) \left({n-r\atop
j-r}\right)c^{r-k}d^{n-i-r+k}.
$$
For the last two equalities we use the following relations between
binomial coefficients, that can be easily verified: $\big({i\atop
h}\big)\big({h\atop k}\big) = \big({i\atop k}\big)\big({i-k\atop
h-k}\big)$ and
$$
\sum_{h=0}^{m+n} \left({m\atop h}\right)\left({n\atop a-h}\right) =
\left({m+n\atop a}\right).
$$
Next, a direct computation gives that the entry $(i,j)$ of
$\phi_n(C_1)$ equals $\big({n-i\atop j-i}\big)$. Therefore the entry
$(i,j)$ of $\phi_n(B)\phi_n(C_1)$ is
$$
\sum_{r=0}^n\sum_{h=0}^n \left({i\atop h}\right)\left({n-i \atop
r-h}\right) a^hb^{i-h}c^{r-h}d^{n-i-r+h} \cdot \left({n-r\atop
j-r}\right)
$$
$$
=  \sum_{h=0}^n \left({i\atop h}\right) a^h b^{i-h} \sum_{r=0}^n
\left({n-i\atop r-h}\right) \left({n-r\atop
j-r}\right)c^{r-h}d^{n-i-r+h}.
$$
We verify that all the entries of $\phi_n(BC_1)$ and
$\phi_n(B)\phi_n(C_1)$ are the same, thus $\phi_n(BC_1) =
\phi_n(B)\phi_n(C_1)$. Next, one can easily verify that $Id_{n+1} =
\phi_n(Id_2)$. For $B=C_1^{-1}$ we then get that $Id_{n+1} =
\phi_n(Id_2) = \phi_n(C_1^{-1})\phi_n(C_1)$ which immediately
implies $\phi_n(C_1^{-1}) = \phi_n(C_1)^{-1}$ and $\phi_n(BC_1^{-1})
= \phi_n(B)\phi_n(C_1^{-1})$.

The verification of $\phi_n(BC_2) = \phi_n(B)\phi_n(C_2)$ is
straightforward. Since $\SL_2(\ZZ)$ is generated by the matrices
$C_1$ and $C_2$ (see~\cite[Section VII.1]{serre_1973} for the
proof), we immediately derive that $\phi_n$ is a homomorphism from
$\SL_2(\ZZ)$ to $\SL_{n+1}(\ZZ)$.

The monoid $M_{n+1,n+1}(\ZZ)$ is generated by matrices $C_1$, $C_2$
and $C_3$ with arbitrary $e$. The equation $\phi_n(BC_3) =
\phi_n(B)\phi_n(C_3)$ is also rather straightforward and then the
first statement of the proposition follows. \endproof

\begin{lemma}
Let $\vp\in\ZZ^{n+1}$ be a good rational approximation to
$\vf(\xi):= (\xi,\xi^2, \ldots, \xi^n)$. Then
$\vq:=\phi_n\left(\begin{array}{cc}a&b\\c&d\end{array}\right)\vp$ is
a good rational approximation to
$\vf\left(\frac{a\xi+b}{c\xi+d}\right)$. More precisely, if
$$
\max_{i\le 1\le n} |p_0\xi^i - p_i| \le p_0^{-\lambda}
$$
then
$$
||q_0||\ll |c\xi+d|^n p_0 + ||B||^n p_0^{-\lambda}\quad\mbox{and}
\quad \max_{1\le i\le n} \left |q_0
\left(\frac{a\xi+b}{c\xi+d}\right)^i - q_i\right| \ll
||B||^np_0^{-\lambda}.
$$
\end{lemma}

\proof We are given that $p_0 \xi^i - p_i =: \delta_i$ with
$|\delta_i|\ll p_0^{-\lambda}$. Therefore $\vp = p_0\vxi + \vdelta$
where $\vxi = (1,\xi, \ldots, \xi^n)$ and $||\vdelta||\ll
p_0^{-\lambda}$. By the construction of $\phi_n$ we have
$$
\phi_n\left(\begin{array}{cc}a&b\\c&d\end{array}\right) p_0\vxi =
\left(\begin{array}{c}
p_0(c\xi+d)^n\\
p_0(c\xi+d)^{n-1}(a\xi+b)\\
\cdots\cdots\cdots\\
p_0(a\xi+b)^n
\end{array}\right) =: \vq^*
$$
and $q_0^* \left(\frac{a\xi+b}{c\xi+d}\right)^i - q_i^* = 0$. Let
$\vq = \phi_n(B)\vp = \vq^* + \vepsilon$ where $\vepsilon =
\phi(B)\vdelta$. Then
$$
\max_{1\le i\le n}\left|q_0\left(\frac{a\xi+b}{c\xi+d}\right)^i -
q_i\right| = \max_{1\le i\le
n}\left|\epsilon_0\left(\frac{a\xi+b}{c\xi+d}\right)^i -
\epsilon_i\right| \ll ||B||^n ||\vdelta||_\infty \ll ||B||^n
p_0^{-\lambda}.
$$

Now we estimate $q_0$: $|q_0|\ll |p_0(c\xi +d)^n| + ||B||^n
||\vdelta||_\infty \ll |c\xi+d|^n p_0 + ||B||^n p_0^{-\lambda}$.
\endproof

From the last Lemma we see that the smaller the expression
$(c\xi+d)$ in terms of $||B||$ is, the closer the point
$\phi_n(B)\vp$ is to the Veronese curve. In particular, if $-d/c$ is
a convergent to $\xi$ and $||B||\asymp |c|\ll
p_0^{\frac{1+\lambda}{2n}}$ then we have
$$
q_0\ll p_0||B||^{-n},\quad \max_{1\le i\le
n}\left|q_0\left(\frac{a\xi+b}{c\xi+d}\right)^i - q_i\right| \ll
||B||^{(1-\lambda)n}q_0^{-\lambda}.
$$

Another simple application of the above lemma is the following
\begin{corollary}
For any $\xi\in\RR$ and $a,b,c,d\in\ZZ$ with $ad\neq bc$, one has
$$
\lambda_n(\xi) =
\lambda_n\left(\frac{a\xi+b}{c\xi+d}\right).\footnote{If the reader
is familiar with the notion of the uniform Diophantine exponent
$\hat{\lambda}_n(\xi)$ they can also verify the corollary for it as
well.}
$$
\end{corollary}

While this corollary is not deep, the author did not see it anywhere
in the literature.

As the  next step, we investigate what happens with linear subspaces
under the map $\phi_n$. Let $\va\in \ZZ^{h+1}$. Consider a subspace
$L_\va\subset \RR^{n+1}$ defined by the equations~\eqref{def_la}.

\begin{proposition}\label{prop4}
Let $B$ be an invertible matrix in $\mathrm{M}_{2,2}(\ZZ)$. Then we
have the following equation
\begin{equation}
\phi_n(B)L_\va = L_{(\phi_h(B^{-1}))^T\va}.
\end{equation}
\end{proposition}

\begin{proof}
Since $B$ is invertible, by Proposition~\ref{prop2} we get that
$\phi_n(B)$ is also an invertible matrix. Therefore $\phi_n(B)L_\va$
is a linear subspace of the same dimension as $L_\va$, i.e. its
dimension is $h$.

We extend the Veronese curve $\VVV_n$ to the complex space
$\CC^{n}$. We also convert $\CC^{n+1}$ to a projective space and
embed it in $\CC^n$ in the standard way:
$$
\tau(\vx):= \left(\frac{x_1}{x_0}, \frac{x_2}{x_0}, \ldots,
\frac{x_n}{x_0}\right).
$$
Then there are exactly $h$ points of intersection of $\VVV_n$ with
$\tau(L_\va)$, counting multiplicities. They are all of the form
$\vf(\xi)$, where $\xi$ are roots of the polynomial $P_\va(x):= a_0
+ a_1x + \ldots + a_hx^h$.

Denote $\vxi:=(1,\xi,\ldots, \xi^n)^T$. Notice that
$$
\phi_n(B)\vxi = \left((c\xi+d)^n, (c\xi+d)^{n-1}(a\xi+b),\cdots,
(a\xi+b)^n\right),
$$
therefore $\tau(\phi_n(B)\vxi) = \vf(\eta)$, where $\eta =
\frac{a\xi+b}{c\xi+d}$. We get that $\VVV_n$ intersects
$\tau(\phi_n(B)L_\va)$ at at least $h$ points $\vf(\eta)$, counting
multiplicities. On the other hand, there is an $h$-dimensional
subspace $L_\vb$ such that $\tau(L_\vb)$ intersects $\VVV_n$ at the
same points. It corresponds to a polynomial $P_\vb(x)$ whose roots
are $\eta$.

Finally, we compute $\vb$. We need to have
$$
\sum_{i=0}^h b_i(a\xi+b)^i (c\xi+d)^{h-i} = 0.
$$
By expanding the brackets and collecting terms at each power of
$\xi$, we get the following system of linear equations
$$
(\phi_h(B))^T \vb = \va\quad\Longleftrightarrow\quad \vb =
(\phi_h(B)^{-1})^T \va = (\phi_h(B^{-1}))^T\va.
$$
\end{proof}

One of the important outcomes of the above proof is that if $\xi$ is
a root of the polynomial $P_\va$ of degree $n$ then
$\frac{a\xi+b}{c\xi+d}$ is a root of the polynomial
$P_{\phi_n(B^{-1})^T\va}$. In the case $\det B = ad-bc = \pm 1$,
this expression can be slightly simplified: $P_{\phi_n(B^*)^T\va}$,
where $B^*:= \begin{pmatrix}
d&-b\\
-c&a
\end{pmatrix}$.

\section{Counting cubic polynomials with bounded height and discriminant}

Now we focus on the case $n=3$. For a given polynomial $P(x) =
c_0+c_1x+c_2x^2+c_3x^3$ we introduce the following height:
$$
H_d(P):=\max\{|c_2|, |c_3|, |c_1c_2|^{1/2}, |c_0c_2^3|^{1/4},
|c_0c_3|^{1/2}, |c_1^3c_3|^{1/4}, |c_0c_1c_2c_3|^{1/4}\}.
$$
Let $x$ be any root of $P$. If either $|c_3|$ or $|c_2|$ has the
largest absolute value among all the coefficients of $P$ then we use
Cauchy's inequality to get $|x|\ll H_d(P)/|c_3|$. If $|c_1|$ has the
largest absolute value then we compute
$$
\left|\frac{c_0}{c_3}\right|^{1/3}\le
\left|\frac{c_1}{c_3}\right|^{1/3} <
\left|\frac{c_1}{c_3}\right|^{1/2} < \frac{|c_1^3c_3|^{1/4}}{|c_3|}
\le \frac{H_d(P)}{|c_3|}.
$$
Finally, if $|c_0|$ has the maximal absolute value then
$$
\max\left\{\left|\frac{c_0}{c_3}\right|^{1/3},\left|\frac{c_1}{c_3}\right|^{1/2}\right\}
\le \left|\frac{c_0}{c_3}\right|^{1/2} =
\frac{|c_0c_3|^{1/2}}{|c_3|} \le \frac{H_d(P)}{|c_3|}.
$$
In the last two cases we use Lagrange-Zassenhaus
inequality~\cite[Lecture VI, Lemma 5]{yap_1999} to compute the upper
bound
$$
|x| \le
2\max\left\{\left|\frac{c_2}{c_3}\right|,\left|\frac{c_1}{c_3}\right|^{1/2},
\left|\frac{c_0}{c_3}\right|^{1/3}\right\}\ll \frac{H_d(P)}{|c_3|}.
$$
We conclude that in all cases one has
\begin{equation}\label{eq11}
|x|\ll \frac{H_d(P)}{|c_3|}.
\end{equation}
By examining the formula for $D(P) = c_1^2c_2^2 -4c_0^3c_2
-4c_1c_3^3-27c_0^2c_3^2+18c_0c_1c_2c_3$, one verifies that $D(P) \ll
H^4_d(P)$.

For a fixed $\va$ we consider all polynomials $P$ that are
equivalent to $P_\va$, i.e. all polynomials whose roots are
$\mu_B(\xi)$ where $\xi$ are the roots of $P_\va$ and $\mu_B(x) =
\frac{cx+d}{ax+b}$, $\det B=\pm 1$. Denote by $R_\va$ the polynomial
in this equivalence class of the minimal possible height $H_d$.

\begin{lemma}\label{lem11}
There exists an absolute constant $\epsilon>0$ such that for all
$\va\in\ZZ^4$ with implied $P_\va$ of non-zero discriminant, the
distance between any two roots $x_1,x_2$ of $R_\va$ is bigger than
$\epsilon$.
\end{lemma}

\proof Let $x_1,x_2,x_3$ be the roots of $R_\va$ and $c_0,c_1,c_2,c_3$ be its
coefficients. Suppose the contrary: $|x_1-x_2|<\epsilon$ for some small
enough $\epsilon>0$. If at least one of $x_1,x_2$ is real then without loss
of generality we can assume that $x_1\in \RR$. Otherwise, $x_1$ and $x_2$ are
conjugates to each other. Notice that by Lagrange's bound on the roots,
$|x_3-x_1|<\frac{H_d(R_\va)}{|c_3|}$.

By Proposition~\ref{prop4}, the polynomial $Q =
(c-ax)^3R_\va\circ\mu^{-1}_B$ with roots $\mu_B(x_1), \mu_B(x_2)$
and $\mu_B(x_3)$ has coefficients
\begin{equation}\label{eq12}
\begin{array}{l}
c_0(\mu) = c^3 c_0-c^2dc_1 + cd^2c_2-d^3c_3, \\
c_1(\mu) = -3ac^2 c_0 + (2acd+bc^2)c_1 - (2bcd+ad^2)c_2 + 3bd^2c_3,\\
c_2(\mu) = 3a^2cc_0 - (2abc+a^2d)c_1 + (2abd+b^2c)c_2 - 3b^2dc_3,\\
c_3(\mu) = -a^3c_0 + a^2bc_1-ab^2c_2+b^3c_3.\\
\end{array}
\end{equation}

Let $-b/a$ be the convergent of $\Re(x_1)$ where $a$ is the largest
possible denominator that satisfies $a\le \epsilon^{-1/2}$. Then
$|x_1+b/a|\le 2\epsilon^{1/2}a^{-1}$ which in turn implies
$|x_2+b/a|\ll \epsilon^{1/2}a^{-1}$. Next, we have
$\epsilon^{1/2}a^{-1} < 1\le \frac{H_d(R_{\va})}{|c_3|}$ and
therefore
$$|R'_\va(-b/a)|\ll
|c_3|\epsilon^{1/2}a^{-1}(|x_3+b/a|+\epsilon^{1/2}a^{-1})\stackrel{\eqref{eq11}}\ll
\epsilon^{1/2}a^{-1} H_d(R_\va).$$

Let $-d/c$ be the previous convergent of $\Re(x_1)$ before $-b/a$.
The same calculations as before lead to
$$
\left|x_1+\frac{d}{c}\right|\le (ac)^{-1},\quad
\left|x_2+\frac{d}{c}\right|\ll (ac)^{-1},\quad |R'_\va(-d/c)|\ll
(ac)^{-1}H_d(R_\va).
$$

Now we are ready to estimate the coefficients of $Q$:
$$
|c_3(\mu)| = a^3|c_3|
\left|\frac{b}{a}+x_1\right|\left|\frac{b}{a}+x_2\right|\left|\frac{b}{a}+x_3\right|
\ll a\epsilon H_d(R_\va) < \epsilon^{1/2} H_d(R_\va);
$$
$$
|c_2(\mu)| = \left|-3a^2c R_\va\left(-\frac{b}{a}\right) +(ad-bc) a
R'_\va\left(-\frac{b}{a}\right)\right|\ll \epsilon^{1/2} H_d(R_\va).
$$
Analogous computations for the other two coefficients give
$|c_1(\mu)|\ll H_d(R_\va)$ and $|c_0(\mu)|\ll H_d(R_\va)$. Finally,
straightforward computations verify that $H_d(Q) \ll
\epsilon^{1/8}H_d(R_\va)$. But that contradicts the choice of
$R_\va$ for $\epsilon$ small enough.
\endproof

Without loss of generality suppose that the distance between $x_2$ and $x_3$
is the largest among the distances between three roots $x_1,x_2,x_3$ or
$R_\va$. Then at least one of the other distances will be larger than
$\frac12 |x_2-x_3|$. Without loss of generality, let the shortest distance
between this numbers be $|x_1-x_2|\asymp 2^d$, $d\ge 0$. More exactly, if
$|x_1-x_2|<1$ we define $d=0$. Otherwise let $d$ be such that $2^d\le
|x_1-x_2|<2^{d+1}$. Then Lemma~\ref{lem11} implies that
\begin{equation}\label{eq16}
|D(R_\va)| \gg 2^{2d}|c_3(x_2-x_3)|^4.
\end{equation}
The same inequality is obviously true for all other pairs of roots
$x_i$ and $x_j$, $i\neq j,\; i,j\in\{1,2,3\}$.


\begin{proposition}\label{prop5}
For any $\epsilon>0$ there exists a constant $c = c(\epsilon)$ such
that for any cubic polynomial $P_\va$, the cardinality $N(P_\va, H)$
of the set
$$
\{P\in \ZZ[x]: P\approx P_\va,\quad H(P)\le H\}
$$
satisfies
\begin{equation}\label{prop5_eq}
N(P_\va,H) \le cH^{2/3+\epsilon} |D(P_\va)|^{-1/6}.
\end{equation}
\end{proposition}

{\bf Remark.} Notice that here we do not impose any restrictions on
$P_\va$. In particular, it can be reducible. If $D(P_\va) = 0$ then
the upper bound in~\eqref{prop5_eq} becomes infinity, i.e. in this
case the proposition is trivial.

\proof Let $R_\va = c_0+c_1x+c_2x^2+c_3x^3$ be the polynomial with
the minimal height $H_d$ among all polynomials in the equivalence
class of $P_\va$. Let $x_1,x_2,x_3$ be its roots. Then
Lemma~\ref{lem11} implies that $|x_i - x_j| \gg 1$, for all $i\neq
j; i,j\in\{1,2,3\}$. In particular, we have $|D(P_\va)|=
|D(R_\va)|\gg |c_3|^4$. The aim is to compute an upper bound for the
number of M\"obius transforms $\mu$ such that $H((c-ax)^3R_\va\circ
\mu^{-1})\le H$.

Notice that there are only finitely many M\"obius transforms $\mu_B$
with $\det B = \det\begin{pmatrix}a&b\\c&d\end{pmatrix}=\pm 1$ such
that $|a|\ge|c|$ and $|b|<|d|$. Also, M\"obius transforms come in
pairs $\mu_B$ and $\mu_{\tilde{B}}$ where $\tilde{B} =
\begin{pmatrix}c&d\\a&b\end{pmatrix}$. Therefore without loss of generality we may assume that $\mu(x) =
\frac{cx+d}{ax+b}$ where $|a|\ge |c|$ and $|b|\ge |d|$. Then for
each pair $a, b$ there exist at most 4 M\"obius transforms with all
the required properties and fixed entries $a,b$. Therefore it is
sufficient to compute an upper bound for the number of pairs $a, b$
that may lead to a polynomial $P$ with $H(P)\ll H$. Denote the set
of such pairs by $M(P_\va, H) = M$.

Given $x\in\RR$, $Q\in \RR^+$ and $0\le t<2$, recall that the number
$N(Q,t)$ of pairs $a,b$ such that $|a|\le Q, |ax-b|< a^{1-t}$
satisfies
$$
N(Q,t)\ll Q^{2-t},
$$
while for $t=2$, $N(Q,t)\ll \log Q$. Indeed, for $t=2$, all such
solutions $b/a$ come from convergents or semiconvergents of $x$.
While for $t<2$ one notices that the distance between two solutions
of the inequality  $|b_1/a_1 - b_2/a_2|$ is at least $1/Q^2$. Hence
the number of such solutions with $Q/2\le a$ is bounded from above
by
$$
2\left(\frac{Q}{2}\right)^{-t} /\; Q^{-2} = 2^{1+t} Q^{2-t}\ll
Q^{2-t}.
$$
Iterating this process for $Q/2, Q/4,\ldots$ proves the claim.

Fix $1>\epsilon>0$, $2> t\ge 0$ and consider the set $S(t,\epsilon)$
of points $a,b\in \ZZ^2$ such that
$$
|a|^{-t-\epsilon}<\left|x_1 + \frac{b}{a}\right|\le |a|^{-t}
$$
for the closest root $x_1$ of $R_\va$ to $-b/a$. Since all the roots
of $R_\va$ are placed far apart from each other, we must have $|x_j
+ b/a|\gg 1$, $j\in\{2,3\}$ for the other two roots. Then
from~\eqref{eq12} we compute
$$
|c_3(\mu)| = |a^3R_\va(-b/a)| \gg |a^3c_3
(x_1-x_2)(x_1-x_3)(x_2-x_3)|\cdot \frac{|x_1+b/a|}{|x_2-x_3|} \gg
\left| \frac{a^{3-t-\epsilon}|D(P_\va)|^{1/2}}{c_3(x_2-x_3)}\right|.
$$
Since we must have $|c_3(\mu)|\le H$, this establishes an upper
bound on the size $|a|$:
\begin{equation}\label{eq17}
|a|\ll \left(\frac{H
|c_3(x_2-x_3)|}{|D(P_\va)|^{1/2}}\right)^{\frac{1}{3-t-\epsilon}}\stackrel{\eqref{eq16}}\ll
\left(\frac{H}{|2^{d/2}D(P_\va)|^{1/4}}\right)^{\frac{1}{3-t-\epsilon}}
.
\end{equation}
Finally, we get that for $t<2$ the number of points $(a,b)$ in
$S(t,\epsilon)$ that satisfy $H((c-ax)^3R_\va\circ\mu^{-1})\le H$ is
bounded from above by
$$
a(t,\epsilon) \ll
\left(\frac{H}{|D(P_\va)|^{1/4}}\right)^{\frac{2-t}{3-t-\epsilon}}\ll
H^{\frac{2}{3}+\epsilon_1} |D(P_\va)|^{-1/6},
$$
where $\epsilon_1 = \frac{2}{3-\epsilon} - \frac{2}{3}$.

Next, consider the set $S(2,1/2)$. In this case, the values $a$ are
bounded from above by $(H|D(P_\va)|^{-1/4})^2$. Hence the number of
points $(a,b)$ in this set that satisfy
$H((c-ax)^3R_\va\circ\mu^{-1})\le H$ is bounded from above by
$$
a(2,1/2)\ll
\max\left\{1,\log\left(\frac{H}{|D(P_\va)|^{1/4}}\right)\right\}.
$$
Notice that for $N(P_\va,H)$ to be nonzero, we must have $H\gg
|D(P_\va)|^{1/4}$ because all the polynomials of height at most $H$
have the discriminant at most $54H^4$. Under these conditions we
have that $a(2,1/2)\ll H^{2/3}D(P_\va)^{-1/6}$.

Consider the pairs $a,b$ such that
$$
\left|x_1 + \frac{b}{a}\right|\le |a|^{-5/2}.
$$
From Roth's theorem we know that the number of such pairs $a,b$ is
finite. Moreover,~\cite[Theorem 1]{dav_rot_1955} gives the following
upper bound for the number of such pairs:
$$
a(5/2)\ll \log\log H\ll  H^{2/3}D(P_\va)^{-1/6}.
$$
Denote the set of the pairs $(a,b)$ with this property and such that
$H((c-ax)^3R_\va\circ\mu^{-1})\le H$ by $S(5/2)$.

The remaining pairs $a,b$ satisfy $|x_i + a/b|>1$ for all
$i\in\{1,2,3\}$. Split them into subsets $S_0(k)$ where every
$(a,b)\in S_0(k)$ satisfies
$$
2^k\le \left|x_i + \frac{b}{a}\right|<2^{k+1}.
$$
The number of pairs in this set with $|a|<Q$ equals $\asymp2^k Q^2$.
First consider $k\le d$. In this case the analogous inequality
to~\eqref{eq17} for $a\in S_0(k)\cap M$ is
$$
|a|\ll
\left(\frac{H}{2^{k+d/2}|D(P_\va)|^{1/4}}\right)^{\frac{1}{3}} .
$$
This implies
$$
\# (S_0(k)\cap M) \ll 2^{\frac{k-d}{3}}H^{\frac23} |D(P_\va)|^{-1/6}.
$$
Next, let $d<k\le D$ where the largest distance between the roots
$x_1,x_2,x_3$ is $\asymp 2^D$. Then we have that $|x_1-x_2|\asymp
2^d$, $|x_1-x_3|\asymp |x_2-x_3|\asymp 2^D$ and $D(P_\va) \asymp
c_3^4 2^{4D+2d}$. If $x_1$ or $x_2$ is the closest root to $-b/a$
then
$$
|c_3(\mu)|\asymp |a^3c_32^{2k+D}| \asymp |a^3 D(P_\va)^{1/4} 2^{2k-d/2}|
$$
and
$$
|a|\ll
\left(\frac{H}{2^{2k-d/2}|D(P_\va)|^{1/4}}\right)^{1/3}\qquad\Longrightarrow\qquad
\#(S_0(k)\cap M)\ll 2^{\frac{d-k}{3}} H^{\frac23}|D(P_\va)|^{-1/6}.
$$
If $x_3$ is the closest root to $-b/a$ then analogous computations give
$$
|c_3(\mu)|\asymp |a^3c_32^{2D+k}| \asymp |a^3 D(P_\va)^{1/4} 2^{k +
D -d/2}|\ge  |a^3 D(P_\va)^{1/4} 2^{2k-d/2}|
$$
and the same bound for $\#(S_0(k)\cap M)$ holds. Finally, for $k>D$,
$|c_3(\mu)|\asymp |a^3 D(P_\va)^{1/4} 2^{3k - D-d/2}|\ge |a^3
D(P_\va)^{1/4} 2^{2k-d/2}|$ and hence we get the same inequality for
$\#(S_0(k)\cap M)$.

To finish the proof of the proposition, we split the interval $[0,2]$ into
$N$ subintervals of equal length $\epsilon$. Then we split the set
$M(P_\va,H)$ into subsets
$$
\bigcup_{i=0}^N (S(i\epsilon, \epsilon)\cap M) \bigcup (S(2,1/2)\cap M)\bigcup
(S(5/2)\cap M) \bigcup_{k=0}^\infty (S_0(k)\cap M) .
$$
By the estimates from above, the total number of points in this
union is bounded from above by
$$
\left(N+2 + 2\sum_{k=0}^\infty 2^{-k/3}\right)H^{\frac23 +
\epsilon_1} D(P_\va)^{-1/6}.
$$
Here $\epsilon_1$ can be taken arbitrary small and $N =
2\epsilon^{-1}$.
\endproof

Note that Proposition~\ref{prop5} strengthens the
result~\cite{davenport_1961} of Davenport from 1961, where he got
$N(P_\va, H)\ll HD(P_\va)^{-1/4}$ for irreducible polynomials
$P_\va$. Now we are ready to prove Theorem~\ref{th3}.

\textsc{Proof of Theorem \ref{th3}.} Let $h(d)$ be the number of
equivalence classes of cubic polynomials that share the discriminant
$d\neq 0$. For convenience of notation we set $h(0)=0$.
Davenport~\cite{davenport_1951} showed that
$$
\sum_{d=-D}^D h(d) \asymp D.
$$
By Proposition~\ref{prop5}, for any given polynomial $P$ of
discriminant $d$ we have at most $c(\epsilon) H^{\frac23 +\epsilon}
d^{-1/6}$ polynomials that are equivalent to $P$ and have the height
at most $H$. Summing over all such polynomials, we get that
$$
N(H,D) \le c H^{\frac23 + \epsilon}\sum_{d = -D}^D h(d)d^{-\frac16}.
$$
Using Abel's summation formula finishes the proof.
\endproof

\section{Theorems~\ref{th2} and \ref{th5} for $n=3$}\label{sec7}

Now we are ready to compute an upper bound for $\dim
S_{3,3}^{\eta,\delta}(I,\lambda)$. We split the set
$A_3^{\eta,\delta}(I,\lambda, k)$ into two subsets:
$A_3^{1,\eta,\delta}(I,\lambda, k)$ consists of all vectors $\va\in
A_3^{\eta,\delta}(I,\lambda,k)$ such that $D(P_\va)\neq 0$, i.e.
$$
A_3^{1,\eta,\delta}(I,\lambda,k):=\{\va\in
A_3^{\eta,\delta}(I,\lambda, k): D(P_\va)\neq 0\}.
$$
$A_3^{2,\eta,\delta}(I,\lambda,k)$ in turn consists of all remaining
elements of $A_3^{\eta,\delta}(I,\lambda,k)$:
$$
A_3^{2,\eta,\delta}(I,\lambda,k):=\{\va\in
A_3^{\eta,\delta}(I,\lambda, k): D(P_\va)= 0\}.
$$
Then the sets $J_3^{\eta,\delta}(I,\lambda,k)$ and
$S_{3,3}^{\eta,\delta}(I,\lambda)$ also split into two subsets
$J_3^{1,\eta,\delta}(I,\lambda,k)$,
$J_3^{2,\eta,\delta}(I,\lambda,k)$ and
$S_{3,3}^{1,\eta,\delta}(I,\lambda)$,
$S_{3,3}^{2,\eta,\delta}(I,\lambda)$ respectively.

We first focus on the set $S_{3,3}^{1,\eta,\delta}(I,\lambda)$.

{\bf The case $\eta\ge 0$}. By definition~\eqref{def_aned}, we have
that all $\va\in A_3^{\eta,\delta}(I,\lambda,k)$ satisfy
$H(P_\va)\ll H = Q^{\lambda-\eta-\delta+\epsilon}$. On top of
that,~\eqref{eq15} implies $D(P_\va)\ll H(P_\va)^4Q^{-1-\lambda
+2\eta}$. Then by Theorem~\ref{th3} the number of such polynomials
and therefore $\#A_3^{1,\eta,\delta}(I,\lambda,k)$ is bounded from
above by
\begin{equation}\label{eq18}
H^{\frac23+\epsilon} \cdot (H^4Q^{-1-\lambda + 2\eta})^{\frac56} =
H^{4+\epsilon} Q^{-\frac56(1+\lambda - 2\eta)}.
\end{equation}
Recall that $Q = 2^k$, $H = Q^{\lambda-\eta-\delta+\epsilon} =
2^{(\lambda-\eta-\delta+\epsilon)k}$. Then one can easily check that
the notion~\eqref{eq18}, as a function of $\eta$, maximises at
$\eta=0$ and equals $Q^{\frac{19\lambda-5}{6} - 4\delta
+\epsilon_1}$ where $\epsilon_1\to 0$ as $\epsilon\to 0$.

On the other hand, there is also a natural upper bound
$Q^{\frac{1+\lambda}{2}}$ for $\#A_3^{1,\eta,\delta}(I,\lambda,k)$
(basically, it is the range for variable $m$ in sets $Q_3(I,\lambda,
k,m)$). One can easily check that this bound is smaller
than~\eqref{eq18} for
$$
\delta< \frac{2\lambda - 1}{3} + \frac{\epsilon_1}{4}.
$$

Recall that for any $J\in L_3^{\eta,\delta}(I,\lambda,k)$ the number
of rational points $\vq$ from the union of all type 3 sets
$Q_3(I,\lambda,k,m)$ such that $R(\va)\subset J$ is bounded from
above by~\eqref{eq27}. Let
$$
Q_3^{\eta,\delta}(I,\lambda,k):= \{\vq\in Q_3(I,\lambda,k) : \exists
J\in L^{1,\eta,\delta}_3(I,\lambda,k) \mbox{ s.t. }R(\vq)\in J\}.
$$
Then we have
$$
S_{3,3}^{1,\eta,\delta}(I,\lambda) = \limsup_{k\to\infty}
\bigcup_{\vq\in Q_3^{\eta,\delta}(I,\lambda,k)} R(\vq).
$$
Then the Hausdorff $s$-series that corresponds to the standard cover
of this limsup set is
$$
\sum_{k=1}^\infty
2^{\left(\frac{19\lambda-5}{6}-4\delta+\epsilon_1\right)k +
\left(\frac{3-5\lambda}{2} + 2\delta+2\epsilon\right)k -
(1+\lambda)ks}\qquad\mbox{if }\quad \delta\ge \frac{2\lambda-1}{3} +
\frac{\epsilon_1}{4}
$$
and
$$
\sum_{k=1}^\infty 2^{\frac{1+\lambda}{2}k  +
\left(\frac{3-5\lambda}{2} + 2\delta+2\epsilon\right)k -
(1+\lambda)ks}\qquad\mbox{if }\quad \delta< \frac{2\lambda-1}{3} +
\frac{\epsilon_1}{4}.
$$
This series converges as soon as the degrees of each power of 2 are
negative, i.e.
$$
s > \left\{\begin{array}{ll} \displaystyle\frac{\frac{2\lambda+2}{3}
- 2\delta+\epsilon_2}{1+\lambda}&\mbox{if}\quad  \delta\ge
\frac{2\lambda-1}{3} + \frac{\epsilon_1}{4}\\[2ex]
\displaystyle \frac{2-2\lambda + 2\delta +
2\epsilon}{1+\lambda}&\mbox{if}\quad \delta< \frac{2\lambda-1}{3} +
\frac{\epsilon_1}{4},
\end{array}\right.
$$
where $\epsilon_2$ tends to zero as $\epsilon\to 0$. Notice that the
first expression, as a function of $\delta$, monotonically
decreases, while the second one monotonically increases. We also
observe that for $\lambda<\frac12$ the second case never happens,
therefore the bound for $s$ maximises when $\delta = 0$ and we have,
by letting $\epsilon$ arbitrarily small, that $\dim
S_{3,3}^{1,\eta,\delta}(I,\lambda) \le \frac{2}{3}$. It is smaller
than $\frac{2-2\lambda}{1+\lambda}$ for $\lambda\le \frac12$.

For $\lambda>\frac12$, the bound for $s$ maximises when $\delta =
\frac{2\lambda - 1}{3} + \frac{\epsilon_1}{4}$. In this case we have
$$
s > \frac{4-2\lambda}{3(1+\lambda)} + \epsilon_2.
$$
This bound in turn implies
\begin{equation}\label{eq28}
\dim S_{3,3}^{1,\eta,\delta} (I,\lambda) \le
\frac{4-2\lambda}{3(1+\lambda)}.
\end{equation}

%

{\bf The case $\eta<0$}. Then by~\eqref{def_aned}, all $\va\in
A_3^{\eta,\delta}(I,\lambda,k)$ satisfy $H(P_\va)\ll H =
Q^{\lambda-\delta}$. The cardinality of
$A_3^{1,\eta,\delta}(I,\lambda,k)$ still satisfies~\eqref{eq18}. Let
$J\in L_3^{1,\eta,\delta}(I,\lambda,k)$. The number of values $m$
such that $R(\vq)$ can intersect an interval $J$, where $\vq\in
Q_3(I,\lambda,k,m)$, is bounded from above by $Q^{-\eta}$.
Therefore, the number of values of $m$, such that $R(\vq)$ intersect
with one of the intervals $J\in L_3^{1,\eta,\delta}(I,\lambda,k)$ is
bounded from above by
$$
H^{4+\epsilon}Q^{-\frac56(1+\lambda-2\eta) - \eta} \ll
Q^{(\lambda-\delta)(4+\epsilon) - \frac56(1+\lambda) +\frac23\eta}
$$
Notice that this bound monotonically increases with $\eta$ therefore
it maximises at $\eta=0$. This case has already been investigated
before and gives~\eqref{eq28}.

%

Consider now the set $S_{3,3}^{2,\eta,\delta}(I,\lambda)$. For all
vectors $\va\in A_3^{2,\eta,\delta}(I,\lambda,k)$ the polynomials
$P_\va$ have zero discriminant which means that they are of the form
$P_\va(x) = (ax-b)^2(cx-d)$ for some integers $a,b,c,d$. Notice that
the polynomial $P_\va(x) = (ax-b)^3$ is also covered, as we can
choose $a=c, b=d$.

If $x\in J$ for some $J\in J_3^{2,\eta,\delta}(I,\lambda,k)$
then~\eqref{eq13} implies
$$
|P_\va(x)|\ll ||\va||_\infty Q^{-1-\lambda}\ll
||\va||_\infty^{1-\frac{1+\lambda}{\lambda-\delta}}.
$$
Since by the Gel'fond lemma we have $H(P_\va) \asymp
(H(P_1))^2H(P_2)$ and $|P_\va(x)|=|P_1(x)|^2|P_2(x)|$, there must be
$i\in\{1,2\}$ such that
\begin{equation}\label{eq20}
|P_i(x)|\ll H(P_i)^{1-\frac{1+\lambda}{\lambda-\delta}}.
\end{equation}
If for some $x\in D^{2,\eta,\delta}_3(I,\lambda)$ there exist
infinitely many polynomials $P$ that satisfy~\eqref{eq20} then $x\in
D_1(I, \frac{1+\lambda}{\lambda-\delta}-1)$. Otherwise there must be
infinitely many vectors $\va\in A^{2,\eta,\delta}_3(I,\lambda)$ such
that
$$
P_\va(x) = Q_\va(x)^2P(x)\quad\mbox{or}\quad P_\va(x)=
P(x)^2Q_\va(x)
$$
where the polynomial $P(x)$ is fixed and satisfies~\eqref{eq20}. In
this case, by letting $k$ in $A^{2,\eta,\delta}_3(I,\lambda,k)$ to
infinity, we get infinitely many polynomials $Q_\va(x)$ that satisfy
$$
Q_\va(x) < H(Q_\va)^{1-\frac{1+\lambda}{\lambda-\delta} + \epsilon}.
$$
where $\epsilon$ can be made arbitrarily small. We conclude that in
all cases,
$$
D^{2,\eta,\delta}_3(I,\lambda) \subset
D_1\left(I,\frac{1+\lambda}{\lambda-\delta}-1\right).
$$

Finally, by Jarnik-Besicovich theorem,
$$
\dim S_{3,3}^{2,\eta,\delta}(I,\lambda) \le \dim
D^{2,\eta,\delta}_3(I,\lambda)\le \dim D_1(I,1/\lambda) =
\frac{2(\lambda-\delta)}{1+\lambda},
$$
which is not bigger than $\frac{4-2\lambda}{3(1+\lambda)}$ for all
$\delta \ge \frac{4\lambda-2}{3}$. Also notice that for $\lambda\le
\frac12$ this expression is smaller than
$\frac{2-2\lambda}{1+\lambda}$ hence we prove the last case of
Theorem~\ref{th2} for $n=3$.

In the remaining part of this section we assume that
$\lambda>\frac12$ and $\delta\le\sigma_1 := \frac{4\lambda-2}{3}$.
Consider $J\in J_3^{2,\eta,\delta}(I,\lambda,k)$. We must have
either $d/c\in J$ or $b/a\in J$.

{\bf The case $d/c\in J$ but $b/a\not\in J$}. by examining the
derivative $P'_\va(x)$ we find that the largest value of
$|P_\va(x)|$ for $x$ between $b/a$ and $d/c$ is for $x_0 =
\frac{b}{3a}+\frac{2d}{3c}$. One can easily check that $|x_0 -
b/a|\asymp |d/c-b/a|$, therefore there exists $x\in J$ such that
\begin{equation}\label{eq30}
\left|x - \frac{d}{c}\right| \gg Q^{-\frac{1+\lambda}{2} -
\eta}\quad\mbox{and}\quad \left|x - \frac{b}{a}\right| \asymp
\left|\frac{d}{c}-\frac{b}{a}\right|.
\end{equation}
Consider $|P_\va(x)| = H(P_\va)|x-c/d||x-b/a|^2<H(P_\va)
Q^{-1-\lambda}$. Then the above bounds imply
$$
\left|\frac{b}{a} - \frac{d}{c}\right| \ll Q^{-\frac{1+\lambda -
2\eta}{4}}.
$$
If $c\le H^{1/3}$, where $H$ is an upper bound for $H(P_\va)$ then
the number of fractions $d/c\in J$ and hence the number of
corresponding intervals $J$ is bounded from above by $\ll H^{2/3}$.
Now suppose the contrary $c>H^{1/3}$. Then we have $a\le
(H/c)^{1/2}\le H^{1/3}$. For a fixed rational number $b/a$, the
value of $c$ can change from $1$ to $H/a^2$. Therefore the number of
fractions $d/c$ that satisfy the first inequality in~\eqref{eq30} is
bounded from above by $\ll \max\{1,H^2a^{-4}Q^{-\frac{1+\lambda -
2\eta}{4}}\}$. Summing over all rational fractions with denominator
$a$ and then over all $a$, we end up with the following bound for
the number of interval $J\in J_3^{2,\eta,\delta}(I,\lambda,k)$ which
satisfy $d/c\in J$, $b/a\not\in J$, $c>H^{1/3}$:
$$
\sum_{a=1}^{H^{1/3}} \sum_b \max\left\{1,\frac{H^2}{a^4}
Q^{-\frac{1+\lambda - 2\eta}{4}}\right\}\ll H^{2/3} + H^2
Q^{-\frac{1+\lambda - 2\eta}{4}}.
$$

{\bf The case $b/a\in J$ and $d/c\in J$}. Then we have either $a\le
H^{1/3}$ or $c\le H^{1/3}$, hence the number of such intervals $J\in
J_3^{2,\eta,\delta}(I,\lambda,k)$ is bounded from above by $\ll
H^{2/3}$. Notice that the case $b/a = d/c$ always falls into this
category.

{\bf The case $b/a\in J$ but $d/c\not\in J$}. If $a\le H^{1/3}$ then
the number of corresponding intervals $J$ is again $\ll H^{2/3}$.
Hence assume that $a>H^{1/3}$. Since we have $a^2\le H$ and the
number of possible numerators $b$ such that $b/a\in I$ is $\ll a$,
we have at most $\ll H$ such intervals $J$. We also have that for
$|x-\frac{b}{a}|\le Q^{-\frac{1+\lambda}{2}}$, $|P_\va(x)|\le
HQ^{-1-\lambda}$ therefore in this case we must have $\eta\le0$. On
the other hand, the bound~\eqref{eq9} for $x_0 = b/a$ gives
$$
H^{-2/3}\le \frac{a}{H}\le \frac{1}{ac}\le\left|\frac{d}{c} -
\frac{b}{a}\right| \asymp \frac{|P''_\va(x_0)|}{H(P_\va)} \ll
Q^{2\eta}.
$$
Since for $\eta\le0$ one has $H\ll Q^{\lambda - \delta}$, we derive
that $\eta \gtrsim -\frac{\lambda - \delta}{3}$ or for large enough
$k$, $\eta > -\frac{\lambda - \delta}{3} - \epsilon$. If one does
not impose any conditions on which of $b/a, d/c$ belong to $J$, then
we can still get a lower bound for $\eta$ but a weaker one. Indeed,
for any $x\in J$~\eqref{eq9} gives
$$
H(P_\va)\asymp P'''(x) \ll HQ^{\frac{1+\lambda}{2} + 3\eta},
$$
therefore $\eta\ge -\frac{1+\lambda}{6} - \epsilon$.

We conclude that $\#A_3^{2,\eta,\delta}$ is bounded from above by
\begin{equation}\label{eq29}
H^{2/3} + H^2 Q^{-\frac{1}{4}(1+\lambda - 2\eta)} + H \cdot
\chi_{\left[-\frac{\lambda+\delta}{3} - \epsilon, 0\right]} (\eta),
\end{equation}
where $\chi_I (\eta)$ is the characteristic function of an interval
$I$.


Now we proceed as in the case of $S_{3,3}^{1,\eta,\delta}$. For
$\eta\ge 0$ we have that~\eqref{eq29} is maximised for $\eta=0$.
Then the Hausdorff $s$-series for the standard cover of
$S_{3,3}^{2,\eta,\delta}$ is
$$
\sum_{k=1}^\infty \left(2^{\left(\frac23(\lambda-\delta)\right)k}+
2^{\left(2\lambda-2\delta-\frac{1+\lambda}{4}\right)k}\right)
2^{\left(\frac{3-5\lambda}{2}+2\delta+2\epsilon\right)k -
(1+\lambda)sk}.
$$
This series converges as soon as
$$
s>\max\left\{\frac{9-11\lambda+8\delta}{6(1+\lambda)} + \epsilon_3,
\frac{5-3\lambda}{4(1+\lambda)} +
\epsilon_3\right\}\stackrel{\delta\le \sigma_1}=
\max\left\{\frac{11-\lambda}{18(1+\lambda)},
\frac{5-3\lambda}{4(1+\lambda)}\right\} + \epsilon_3.
$$
Notice that both expressions in the maximum are smaller than
$\frac{4-2\lambda}{3(1+\lambda)}$ for $\lambda<\frac{13}{11}$, hence
in this case we conclude
$$
\dim S_{3,3}^{2,\eta,\delta}(I,\lambda)\le
\frac{4-2\lambda}{3(1+\lambda)}.
$$

Now consider $\eta<0$. Then the number of values of $m$, such that
$R(\vq)$, $\vq\in Q_3(I,\lambda,k,m)$ intersects with one of the
intervals $J\in L_3^{2,\eta,\delta}(I,\lambda,k)$ is bounded from
above by
$$
H^{2/3}Q^{-\eta}+ H^2Q^{-\frac14(1+\lambda-2\eta)-\eta} +
HQ^{-\eta}\cdot \chi_{\left[-\frac{\lambda+\delta}{3} - \epsilon,
0\right]} (\eta).
$$
This expression decreases with $\eta$ hence it is maximised when
$\eta$ is the smallest possible, i.e. $\eta =
-\frac{1+\lambda}{6}-\epsilon$ for the first two terms and $\eta =
-\frac{\lambda+\delta}{3}-\epsilon$ for the last one. Then the
Hausdorff $s$-series for $S_{3,3}^{2,\eta,\delta}(I,\lambda)$ is
$$
\sum_{k=1}^\infty
\left(2^{\left(\frac23(\lambda-\delta)+\frac{1+\lambda}{6}+\epsilon\right)}
+ 2^{\left(2\lambda-2\delta
-\frac{1+\lambda}{6}+\frac{\epsilon}{2}\right)k} + 2^{\left(\lambda
-\delta + \frac{\lambda-\delta}{3} + \epsilon\right)k}\right)
2^{\left(\frac{3-5\lambda}{2} + 2\delta + 2\epsilon\right)k -
(1+\lambda)sk}.
$$
This series converges as soon as
$$
s> \max\left\{\frac{5-5\lambda+4\delta}{3(1+\lambda)}+\epsilon_4,
\frac{4-2\lambda}{3(1+\lambda)} + \epsilon_4,
\frac{9-7\lambda+4\delta}{6(1+\lambda)}+ \epsilon_4 \right\}
$$$$
\stackrel{\delta\le \sigma_1}=
\max\left\{\frac{7+\lambda}{9(1+\lambda)},
\frac{4-2\lambda}{3(1+\lambda)},
\frac{19-5\lambda}{18(1+\lambda)}\right\} + \epsilon_4.
$$
One can check that all terms in the maximum are at most
$\frac{4-2\lambda}{3(1+\lambda)}$ as soon as $\lambda\le \frac57$.

We now exhaust all the cases. By letting $\epsilon\to 0$, we finally
conclude that $\dim S_{3,3}^{2,\eta,\delta}(I,\lambda) \le
\frac{4-2\lambda}{3(1+\lambda)}$. This confirms Theorem~\ref{th5}.

%

\section{Theorem \ref{th2} for arbitrary $n$: dealing with resultants}

Here we prove Theorem~\ref{th2}. Split the set
$A_n^\eta(I,\lambda,k)$ into two subsets:
$A_n^{1,\eta}(I,\lambda,k)$ consists of all $\va$ such that $P_\va$
is an irreducible polynomial of degree $n$; and
$A_n^{2,\eta}(I,\lambda,k)$ consists of all remaining vectors. The
sets $J_n^\eta(I,\lambda, k)$ and $D_n^\eta(I,\lambda)$ split into
two corresponding subsets as well.

For all vectors $\va\in A_n^{2,\eta}(I,\lambda,k)$ the polynomials
$P_\va$ can be written as $P_\va = P_1P_2$ where $\deg P_1$ and
$\deg P_2$ are at most $n-1$. Then proceeding in the same way as for
polynomials with zero discriminant in the case $n=3$, we derive
$$
D^{2,\eta}_n(I,\lambda) \subset D_{n-1}\left(I,\frac{1+\lambda}{\lambda}-1\right).
$$
The Jarnik-Besicovich theorem then implies
$$
\dim D_n^{2,\eta}(I,\lambda) \le \frac{n\lambda}{1+\lambda}
$$
which is smaller than $\frac{2-(n-1)\lambda}{1+\lambda}$ for all $\lambda\le
\frac{2}{2n-1}$.

Consider now an arbitrary $\va\in A_n^{1,\eta}(I,\lambda,k)$. Recall
that for $x\in\RR$ the values $\kappa_i(x)$ are defined as $|x-x_i|
= Q^{-\kappa_i}$, where $x_i$ are the roots of $P_\va$ and
$\kappa_1\ge \kappa_2\ge\ldots\ge \kappa_n$. Let $x$ be the centre
of the interval $J^\eta(\va,k)$ and $x_0$ be the real part of $x_1$.
Notice that for all $i\in\{2,\ldots, n\}$ and $y$ between $x$ and
$x_0$ one has $|y - x_i|\le 2|x-x_i|$ and $|y-x_1|\le |x-x_1|$.
Therefore, by increasing the upper bound~\eqref{eq13} of
$|P_\va(x)|$ in the definition of $A_n^\eta(I,\lambda,k)$ by an
absolute constant, we can assume that $x_0\in J^{\eta}(\va,k)$.

Split each set $A_n^{1,\eta}(I,\lambda,k)$ into finitely many
subsets $A_n^{\eta}(I,\lambda,k,\vkappa)$ where $\vkappa =
(\kappa_0, \ldots, \kappa_n)$ such that one has
$$
\kappa_i-\epsilon <\kappa_i(x_0)\le \kappa_i.
$$

If $x_1$ is real, i.e. $x_0=x_1$ then for any $2\le i\le n$ one has
$|x_1-x_i|= Q^{-\kappa_i(x_0)}\gg Q^{-\kappa_i}$. The inequality~\eqref{eq9}
then implies
$$
a_nQ^{-\frac{1+\lambda}{2} + \eta}\gg  |P'_\va(x_1)| \gg a_n Q^{-\sum_{i=2}^n \kappa_i}
$$
or in other words
\begin{equation}\label{eq21}
\sum_{i=2}^n \kappa_i \gtrsim \frac{1+\lambda}{2} - \eta.
\end{equation}
If $x_1$ is not real, then $x_1$ and $x_2$ are complex conjugates and
$\kappa_1(x_0)=\kappa_2(x_0)$, i.e. we must have $\kappa_1=\kappa_2$. Then
the inequalities
$$
a_n Q^{-1-\lambda+2\eta} \gg |P_\va(x_0)| \gg a_n Q^{-\sum_{i=1}^n \kappa_i}
$$
imply that
$$
2\kappa_2 + \kappa_3+\ldots + \kappa_n \ge 1+\lambda - 2\eta.
$$
In view of the inequalities $\kappa_i\gtrsim 0$,~\eqref{eq21}
follows again.

Consider two vectors $\va_1, \va_2\in
A^{\eta}_n(I,\lambda,k,\vkappa)$. We want to find a suitable lower
bound for $|x_0(\va_1) - x_0(\va_2)|$. We call the corresponding
roots of $P_{\va_1}$ by $x_1, \ldots, x_n$ and set $x_0:=
x_0(\va_1)$. Respectively, we call the corresponding roots of
$P_{\va_2}$ by $y_1,\ldots, y_2$ and set $y_0:=x_0(\va_2)$.

If $|x_0-y_0|= Q^{-\delta}\le Q^{-\kappa_2+\epsilon}$ then for all
$1\le i,j\le n$ we get $|x_i - y_j| \ll Q^{-\kappa_l+\epsilon}$
where $l:=\max \{2,i,j\}$. This together with $|x_i - y_j|\ll 1$
leads to
$$
\prod_{1\le i,j\le n} |x_i - y_j| \ll Q^{- 4\sum_{i=2}^n \kappa_i + 4n\epsilon}.
$$
Now since $P_{\va_1}$ and $P_{\va_2}$ are distinct irreducible
polynomials with integer coefficients, we must have $\res
(P_{\va_1}, P_{\va_2}) \ge 1$. This implies that $||\va_1||_\infty^n
||\va_2||_\infty^n Q^{-4\sum_{i=2}^n \kappa_i + 4n\epsilon}\gg 1$.
For $\eta\ge 0$ we use upper bound~\eqref{def_ane}:
$||\va_i||_\infty\ll Q^{\lambda-\eta+\epsilon}$, $i\in\{1,2\}$,
and~\eqref{eq21} to derive
$$
0\lesssim
2n(\lambda-\eta+\epsilon)-4\left(\frac{1+\lambda}{2}-\eta\right) +
4n\epsilon\le 2(n-1)\lambda - 2 + \epsilon_5.
$$
For $\eta<0$ by~\eqref{def_ane}, we have $||\va_i||_\infty \ll
Q^{\lambda}$, hence the same inequality can be achieved. Finally, we
notice that for $\lambda<\frac{1}{n-1}$ and $\epsilon_5$ small
enough that inequality is impossible and hence we get a
contradiction.

We derive that $|x_0-y_0| = Q^{-\delta}\ge Q^{-\kappa_2+\epsilon}$. In this
case we have
$$
\max\{|x_0-y_1|, |x_0-y_2|\}\ll Q^{-\delta}.
$$

We compute for $i\in\{1,2\}$
\begin{equation}\label{eq23}
|P_{\va_1}(y_i)| = \left| \sum_{j=0}^n \frac{(y_i-x_0)^j}{j!}
P^{(j)}_{\va_1}(x_0)\right| \stackrel{\eqref{eq9}}\ll ||\va_1||_\infty
\left(\sum_{j=0}^n \min\{Q^{-j\delta}, Q^{-j\delta
-\frac{2-j}{2}(1+\lambda) + j\eta}\}\right).
\end{equation}
Notice that for $\delta\ge \frac{1+\lambda}{2}+\eta$ all the terms
on the right hand side are not bigger than $||\va_1||_\infty
Q^{-1-\lambda}$. Then we compute
$$
1\le |\res(P_{\va_1},P_{\va_2})| \ll
||\va_1||_\infty^{n-2}||\va_2||_\infty^n
|P_{\va_1}(y_1)||P_{\va_1}(y_2)|\ll Q^{2n\lambda - 2-2\lambda}
$$
Notice that the last inequality is not possible for $\lambda <
1/(n-1)$. We conclude that one must have $\delta<
\frac{1+\lambda}{2}+\eta$. Under this condition, the second term in
the minimum~\eqref{eq23} increases as a function of $j$, while the
first term always decreases.

Recall that by~\eqref{eq32} and~\eqref{eq31}, the value $\eta$ lies
in the range: $\frac{(1-n)(\lambda+1)}{2(n+1)} \lesssim
\eta<\frac{\lambda+1}{2}$. We split  this interval into smaller
segments. Namely, fix $1\le j\le n$ and consider the case
\begin{equation}\label{eta_class}
\frac{(1-j)(1+\lambda)}{2(j+1)}\le \eta <
\frac{(2-j)(1+\lambda)}{2j}.
\end{equation}
For $j=n$ the left inequality sign should be replaced by $\lesssim$.
Then the largest term in the sum on the right hand side
of~\eqref{eq23} is either
\begin{equation}\label{eq24}
Q^{-j\delta - \frac{2-j}{2}(1+\lambda) + j\eta}\quad\mbox{ or }\quad
Q^{-(j+1)\delta}.
\end{equation}
Notice that the first term is bigger than the second one if
$\delta\gtrsim \frac{2-j}{2}(1+\lambda) - j\eta$.

Consider the case $j=1$. It corresponds to $\eta\ge 0$ and hence
$||\va_i||_\infty \ll Q^{\lambda-\eta+\epsilon}$. If the first term
in~\eqref{eq24} is the largest one we compute
$$
1\le |\res(P_{\va_1},P_{\va_2})| \ll Q^{2n(\lambda-\eta+\epsilon) -
2\delta - 1-\lambda + 2\eta}.
$$
From here we derive
$$
\delta\lesssim \frac{(2n-1)\lambda - 1 - 2(n-1)\eta +
2n\epsilon}{2}.
$$
However, this inequality is incompatible with $\delta\gtrsim
\frac{1+\lambda}{2} - \eta$ for $\lambda<\frac{1}{n-1}$, $n>2$ and
$\epsilon$ small enough. Therefore we get that the second bound
in~\eqref{eq24} should take place. Now analogous computations for
the resultant of $P_{\va_1}$ and $P_{\va_2}$ give
$$
1\ll Q^{2n(\lambda-\eta+\epsilon) - 4\delta}
\quad\Longrightarrow\quad \delta\lesssim
\frac{n(\lambda-\eta+\epsilon)}{2} \quad\Longrightarrow\quad  |x_0-y_0| \gg Q^{-\frac{n(\lambda-\eta+\epsilon)}{2}}.
$$
Since this inequality must be satisfied for all pairs
$\va_1,\va_2\in A_n^\eta(I,\lambda, k,\kappa)$, we have
$\#A_n^\eta(I,\lambda,k,\vkappa)\ll
Q^\frac{n(\lambda-\eta+\epsilon)}{2}$. Compute the $s$-Hausdorff
series that corresponds to the standard cover of
$\limsup_{k\to\infty} J_n^\eta(I,\lambda,k,\vkappa)$:
$$
\sum_{k=1}^\infty 2^{\frac{n\lambda - n\eta + n\epsilon}{2}k}
2^{-\left(\frac{1+\lambda}{2}+\eta-\epsilon\right)sk}.
$$
It converges for
$$
s> \frac{n\lambda -n\eta}{1+\lambda+2\eta} + \epsilon_6.
$$
The right hand side is maximised for $\eta=0$. Taking into account
that $\epsilon_6$ can be made arbitrarily small, we finally get that
\begin{equation}\label{eq22}
\dim \limsup_{k\to\infty} J_n^\eta(I,\lambda,k,\vkappa)\le
\frac{n\lambda}{1+\lambda}.
\end{equation}
This value is smaller than $\frac{2-(n-1)\lambda}{1+\lambda}$ for
$\lambda<\frac{2}{2n-1}$.

Suppose now that $\eta$ satisfies~\eqref{eta_class} for $j>1$. This
automatically means that $\eta<0$, hence $||\va_i||_\infty\ll
Q^\lambda$. We first assume that the first inequality
in~\eqref{eq24} is satisfied.

Estimating the resultant in the same way as before leads to
$$
1\ll Q^{2n\lambda - 2j\delta - (2-j)(1+\lambda)+2j\eta}
\quad\Longrightarrow\quad \delta\lesssim \delta_0 := \frac{(2n-2)\lambda -2
+j(1+\lambda) + 2j\eta}{2j}.
$$
As in the previous case, we get $\#A_n^{\eta}(I,\lambda, k,\kappa) \ll
Q^{\delta_0}$ and the correspondent $s$-Hausdorff series for
$\limsup_{k\to\infty} J_n^\eta(I,\lambda,k,\vkappa)$ is
$$
\sum_{k=1}^\infty 2^{\frac{(2n-2)\lambda - 2 +j(1+\lambda) +
2j\eta}{2j}k} 2^{-\left(\frac{1+\lambda}{2}+\eta-\epsilon\right)sk}.
$$
which converges as soon as
$$
s> \frac{(2n-2)\lambda - 2 +
j(1+\lambda)+2j\eta}{j(1+\lambda+2\eta)}+\epsilon_7.
$$
One can check that for $\lambda<\frac{1}{n-1}$, the last expression
monotonically increases with $\eta$, hence it attains its maximal
value for $\eta = \frac{(2-j)(1+\lambda)}{2j}$. Substituting this
into the above inequality gives $s > \frac{n\lambda}{1+\lambda} +
\epsilon_7$ which leads to~\eqref{eq22}.

Finally, assume that the second bound in~\eqref{eq24} is valid. Then the same
computations lead to
$$
1\ll Q^{2n\lambda - 2(j+1)\delta}\quad\Longrightarrow\quad
\delta\lesssim \frac{n\lambda}{j+1}.
$$
The correspondent $s$-Hausdorff series for $\limsup_{k\to\infty}
\#A_n^\eta(I,\lambda,k,\vkappa)$ converges for
$$
s > \frac{n\lambda}{2(j+1)(1+\lambda+2\eta)}+\epsilon_8.
$$
This value monotonically decreases with $\eta$ therefore it attains it
maximum for $\eta = \frac{(1-j)(1+\lambda)}{2(j+1)}$ which implies the same
lower bound on $s$ as in the previous case. Again, we derive~\eqref{eq22}.

We now exhausted all possible subsets of $D_n^\eta(I,\lambda)$. In each case,
for $\lambda\le \frac{2}{2n-1}$ the Hausdorff dimension of those sets is at
most $\frac{2-(n-1)\lambda}{1+\lambda}$. Hence, Theorem~\ref{th2} is
verified.

\bigskip {\bf Acknowledgement}: I am very thankful to Johannes
Schleischitz for many conversations that helped to produce this
paper. Those conversations were made possible thank to the IVP
program of the Sydney Mathematical Research Institute. I am grateful
to them as well. Finally, I am grateful to the School of Mathematics
and Statistics at the University of Sydney for their constant
support.

%

\bigskip
\noindent Dzmitry Badziahin\\ \noindent The University of Sydney\\
\noindent Camperdown 2006, NSW (Australia)\\
\noindent {\tt dzmitry.badziahin@sydney.edu.au}

\end{document}